\documentclass[10pt]{article}
\usepackage[letterpaper,margin=1.00in]{geometry}

\usepackage{booktabs}
\usepackage{hyperref}
\usepackage{amsmath,amssymb,amsthm,bbold}		
\usepackage{xcolor}
\usepackage{lscape}
\usepackage{empheq}
\usepackage{xcolor}
\usepackage{pdflscape}

\usepackage{todonotes}

\newtheorem{theorem}{Theorem}[section]
\newtheorem{lemma}[theorem]{Lemma}
\newtheorem{corollary}[theorem]{Corollary}

\usepackage{algorithm}
\usepackage{algpseudocode}

\usepackage{graphicx}
\usepackage{epstopdf} 

\usepackage{tikz}                                
\usetikzlibrary{calc}

\usetikzlibrary{shadings}           
\usetikzlibrary{positioning,fit,matrix,shadows,chains,arrows,shapes,spy,fadings}
\usepackage{pgfplots}
\usetikzlibrary{pgfplots.units,shapes.symbols,shapes.arrows,positioning}

\newcommand{\R}{\mathbb{R}}					
\newcommand{\st}{\mbox{s.t.}}				
\newcommand{\suml}{\sum\limits}			

\renewcommand{\d}{\mathrm d}
\newcommand{\D}{\mathrm D}
\renewcommand{\div}{\mathrm{div}}

\newcommand{\hMOO}{g_1^{\mathrm{OMO}}}	
\newcommand{\hMOOhat}{\hat{g}_1^{\mathrm{OMO}}}	
\newcommand{\rlam}{\mathfrak{P}}	

\newcommand{\tini}{0}	
\newcommand{\tend}{T}

\newcommand{\crtlone}{u}
\newcommand{\stateone}{x}
\newcommand{\objone}{J}
\newcommand{\regone}{\alpha}
\newcommand{\dynone}{G}
\newcommand{\lagrOMone}{\lambda}
\newcommand{\lagrMOone}{\Lambda}

\newcommand{\fOMone}{h^{\mathrm{OM}}}

\newcommand{\fMOone}{h^{\mathrm{MO}}}

\newcommand{\crtltwo}{u}
\newcommand{\statetwo}{x}
\newcommand{\objtwo}{J}
\newcommand{\regtwo}{\alpha}
\newcommand{\dyntwo}{G}
\newcommand{\lagrOMtwo}{\lambda}
\newcommand{\lagrMOtwo}{\Lambda}

\newcommand{\fOMtwo}{f^{\mathrm{OM}}}
\newcommand{\fOMtwohat}{\hat{f}^{\mathrm{OM}}}
\newcommand{\fMOtwo}{f^{\mathrm{MO}}}
\newcommand{\fMOtwohat}{\hat{f}^{\mathrm{MO}}}

\newcommand{\crtllead}{v}
\newcommand{\statelead}{\zeta}
\newcommand{\objlead}{J^L}
\newcommand{\reglead}{\beta}

\newcommand{\lagrMOlead}{\Phi}

\newcommand{\crtlfoll}{w}
\newcommand{\statefoll}{\xi}
\newcommand{\objfoll}{J^F}
\newcommand{\regfoll}{\gamma}
\newcommand{\dynfoll}{P}
\newcommand{\lagrOMfoll}{\psi}
\newcommand{\lagrMOfoll}{\Psi}

\newcommand{\fOOM}{g^{\mathrm{OOM}}}
\newcommand{\fOOMhat}{\hat{g}^{\mathrm{OOM}}}
\newcommand{\fMOO}{g^{\mathrm{MOO}}}
\newcommand{\fMOOhat}{\hat{g}^{\mathrm{MOO}}}
\newcommand{\fOMO}{g^{\mathrm{OMO}}}
\newcommand{\fOMOhat}{\hat{g}^{\mathrm{OMO}}}

\begin{document}

\title{Multiscale Control of Stackelberg Games}
\author{Michael Herty\footnote{herty@igpm.rwth-aachen.de, RWTH Aachen University, Germany}, Sonja Steffensen\footnote{steffensen@igpm.rwth-aachen.de, RWTH Aachen University, Germany}, and Anna Th\"unen\footnote{thuenen@igpm.rwth-aachen.de, RWTH Aachen University, Germany}}

\date{\today}
\maketitle

\begin{abstract}

We present a linear--quadratic Stackelberg game with a large number of followers and we also derive the mean field limit of infinitely many followers.
The relation between optimization and mean-field limit is studied and
conditions for consistency are established. 
Finally, we propose a numerical method based on the derived models and present numerical results.



\end{abstract}

\section{Introduction}
Game theory extents classical optimization  by allowing for competing goals of possibly many actors.
Early considerations and economic applications are described in~\cite{Neumann2007}.
A theoretical breakthrough has been made by Nash by formalizing  the concept of equilibrium~\cite{Nash1951}.
Stackelberg extended the models by putting one player in an special position, called the leader~\cite{Stackelberg2011}, establishing the class of Stackelberg games.
In the last decades such multilevel games served as a tool for the analysis of systems of multiple competing interests and hierarchies.
A prominent application is the analysis of electricity markets~\cite{Allevi2017, Henrion2012, Hu2007} using multi-leader follower games.
Other applications include traffic and tolling~\cite{Harks2019,Koh2010} as well as telecommunication~\cite{Nowak2018,Wang2008}.

These applications usually involve modeling large populations of followers.
For example, the demand of all customers for electricity is represented by one single independent system operator (ISO), which currently also provides a precise model for the current practice in energy markets, e.g.~\cite{Aussel2016, Aussel2016a,Aussel2015}.
Many commuters in tolling models are modeled as one unit seeking a Wardrop equilibrium, whereby the interaction of these units does not play a role in road traffic, e.g.~\cite{Harks2019}.
Similarly, in~\cite{Wang2008} the internet providers are modeled as individual leaders but  data traffic is not further adressed.

We are interested in the study of Stackelberg games under possibly infinitely many followers.
Models of interacting agents or followers have been studied e.g. in~\cite{Bellomo2017, Cristiani2014, Pareschi2014}.
In particular, opinion formation and consensus as social models are discussed in~\cite{Hegselmann2002,Motsch2014}.
Other applications include economic and financial market models \cite{Naldi2010} as well as traffic models~\cite{Herty2020}.
Game theoretic foundations of the analysis of these interacting agent systems discussed e.g.~\cite{Lasry2007}.
In \cite{Albi2016} the control of a two-population model is investigated, where a population with a leading role is modeled through the dynamics.
The agents of this leading population are however not leaders in the sense of game theory.

In this paper we consider a Stackelberg game with one leader and a possibly infinite number  of followers.
This population of followers is modeled as a dynamic system.
We are interested in an equilibrium of the game which we characterize by first-order optimality conditions.
We propose the following approach:
The follower level is optimized first with the leader's control as parameter, then the leader's problem is solved  provided certain regularity assumptions hold true, see e.g.~\cite{Aussel2018}.
The limit to infinitely many followers can be derived at different stages of the optimization
 yielding  mean-field descriptions of the model.
The focus of this article is the analysis of the interchangeability of optimization  and derivation of the mean-field, see Figure~\ref{Fig:OrderOfOptMF}.
We establish consistency conditions for the Lagrange multipliers which link the different options.
The novelty of our work lies  not only in the two level problem but also in open loop controls compared, different to prior work studying feedback control techniques, c.f.~\cite{Albi2013,Albi2014,Albi2017,Albi2015a}.
Also, compared to~\cite{Herty2019a}, we derive consistent optimality conditions for a Stackelberg game.

Other related work includes 
a linear quadratic Stackelberg game of a large follower population governed by stochastic differential equations in~\cite{Moon2018}.
There, a local optimal control problem of the followers is solved where the control of the leader is considered as an exogenous stochastic process.
This leads to $\varepsilon$-Nash equilibria and it is shown that $\varepsilon\rightarrow0$ as the number of followers grows to infinity.
A related model is studied in~\cite{Ma2020} where it is distinguished between one major and a large number of minor players.
In contrast to the work in~\cite{Ma2020,Moon2018}, we study a partial differential equation (PDE) on the probability density of the players' states.

Here, we limit ourselves to formal computations in this article.
Other approaches, rigorous derivations, and analytical results on derivatives with respect to measures may be found in~\cite{Bensoussan2017,Cardaliaguet2017,Lasry2007}.

This article is structured as follows:
We begin with the derivation of consistent optimality conditions for two optimal control problems in Section~\ref{Sec:2}.
A model including a single control and a second model where each agent has an individual control, are studied there.
We apply these results to a Stackelberg game in Section~\ref{Sec:3}.
In Section~\ref{Sec:4} a numerical scheme to solve the optimality system of Option 3 is derived.
We conclude with some numerical results in Section~\ref{Sec:5}.

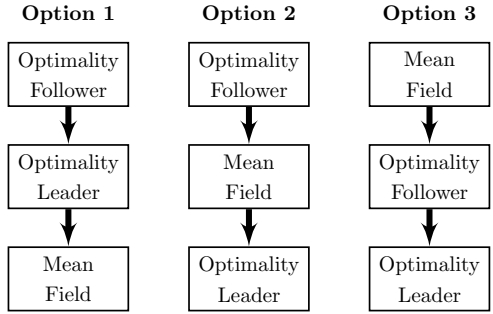
\begin{figure}
	\begin{center}
		\tikzstyle{block4} = [rectangle, draw, thick,
		text width=5em, text centered,  minimum height=3em]
		\tikzstyle{block5} = [rectangle, draw,   thick,
		text width=5em, text centered,  minimum height=3em]
		\tikzstyle{line} = [draw, -latex']
		
	\resizebox{0.4\textwidth}{!}{	
		\begin{tikzpicture}[node distance = 1.7cm, auto]
		\node [] (A0) {\textbf{Option 1}};
		\node [right of=A0,node distance=3cm] (B0) {\textbf{Option 2}};
		\node [right of=B0,node distance=3cm] (C0) {\textbf{Option 3}};
		
		\node [block5,below of=A0,node distance=1cm] (A1) {Optimality\\Follower};
		\node [block5,below of=B0,node distance=1cm] (B1) {Optimality\\Follower};
		\node [block4,below of=C0,node distance=1cm] (C1) {Mean\\ Field};
		\node [block5,below of=A1] (A2) {Optimality\\Leader};
		\node [block4,below of=B1] (B2) {Mean \\Field};
		\node [block5,below of=C1] (C2) {Optimality\\Follower};
		\node [block4,below of=A2] (A3) {Mean \\Field};
		\node [block5,below of=B2] (B3) {Optimality\\Leader};
		\node [block5,below of=C2] (C3) {Optimality\\Leader};


		\path [line,line width=0.1cm] (A1) -- (A2);
		\path [line,line width=0.1cm] (A2) -- (A3);
		\path [line,line width=0.1cm] (B1) -- (B2);
		\path [line,line width=0.1cm] (B2) -- (B3);
		\path [line,line width=0.1cm] (C1) -- (C2);
		\path [line,line width=0.1cm] (C2) -- (C3);
		\end{tikzpicture}
	}
	\end{center}
	\caption{Schematic overview of the various order of operations for the Stackelberg game.}\label{Fig:OrderOfOptMF}
	
\end{figure}

%
%
%

\section{Single Level Problems}\label{Sec:2}
In this section, we study two optimal control problems of interacting agents systems that differ in the nature of the application of the control.
The problem in Section~\ref{Sec:2.1} is controlled by one  control.
In contrast, the problem discussed in Section~\ref{Sec:2.2} captures one control for each agent.

The optimality conditions to each optimal control problem can be derived prior  to the derivation of the mean-field limit or after---resulting in two different optimality systems.
We compare both systems and establish consistency conditions to build a link between them, c.f. Lemma~\ref{Lem:1} and Lemma~\ref{Lem:2}.

Note, the superscript $\mathrm{MO}$ indicates that the mean-field limit is derived prior to optimization. 
The superscript $\mathrm{OM}$ indicates the opposite order, see also Figure~\ref{Fig:OrderOfOptMF}.
\subsection{Single Control  System} \label{Sec:2.1}

We consider the optimal control problem of a system of $N$ interacting agents as follows:
\begin{align}
\begin{split}\label{Eq:OCone}
	\min\limits_\crtlone &~\int\limits_\tini^\tend \left[\objone\left(\crtlone,m\left(\vec{\stateone}\right)\right) +\frac{\regone}{2} \crtlone^2\right]\,\d t\\
	\st &~\dot{\stateone}_i=\frac{1}{N}\suml^N_{j=1} \dynone(\stateone_i,\stateone_j,\crtlone),\quad \stateone_i(\tini)=\stateone_{i,\tini}, \quad i=1,\dots,N,
	\end{split}
\end{align}
where the states $\stateone_i=\stateone_i(t)$ are considered to be in $\R^n$ for the agents $i=1,\dots,N$ and the according initial states are given by $\stateone_{i,\tini}$.
The concatenation of all agents' states is denoted by $\vec{\stateone}=\left(\stateone_i\right)_{i=1}^N\in\R^{Nn}$.
The (common) control is $\crtlone=\crtlone(t)\in\R^{n_\crtlone}$.
The explicit dependence on time $t$  is omitted whenever the intention is clear.
The agents' dynamics is described by $\dynone: \R^n\times\R^n\times\R^{n_\crtlone}\rightarrow\R^n$ and is assumed to be at least differentiable in all arguments.

The common control $\crtlone$ is to be chosen such that an objective functional is minimized over the time horizon $[\tini,\tend]$.
The function $\objone: \R^{n_\crtlone}\times\R^n\rightarrow\R$ takes the control $\crtlone$ and $m$ as arguments and is assumed to be differentiable.
The value $m=m\left(\vec{\stateone } \right)$ is considered be  a vector of a moment of the states, i.e. $m: \R^{Nn}\rightarrow\R^n$ with $$m\left(\vec{\stateone } \right)=\frac{1}{N}\suml^N_{i=1} \tilde{m}\left(\stateone_i\right)$$ and $\tilde{m}: \R^{n}\rightarrow\R^n$.
The objective is regularized by a quadratic term of the control $\crtlone$ with a scalar weighting parameter $\regone>0$.

Assuming the agents are identically independent and the interaction $\dynone$ is symmetric, we compute admissible variations with respect to the mean-field density $\fMOone=\fMOone(t,\stateone)$.
With e.g.~\cite[Proposition 2.1]{Herty2019a}, we have the mean-field evolution equation of the state variables and the objective functional form the following optimal control problem in the strong form:
\begin{align}
\begin{split}\label{Eq:OCMFone}
\min\limits_{\crtlone}	&~\int\limits_\tini^\tend \left[\objone\left(\crtlone,m_{\fMOone}(t)\right) +\frac{\regone}{2} \crtlone^2\right]\,\d t\\
\st &~0=\partial_t\fMOone+\div_\stateone \left(\fMOone \int \dynone(\stateone,\hat{\stateone},\crtlone) \fMOone(t,\hat{\stateone}) \,\d\hat{\stateone}   \right)\\
&~\fMOone(\tini,\stateone)=\fMOone_\tini(\stateone),
\end{split}
\end{align}
where  $m_{\fMOone}(t)=\int \tilde{m}(\stateone) \fMOone\,\d\stateone$.
Formally, the dynamics and the cost of~\eqref{Eq:OCone}  and~\eqref{Eq:OCMFone}
are recovered for  the mean-field probability density $\fMOone$  chosen as empirical distribution:
\begin{equation}\label{Eq:EmpirMeas}
	\mu_N\left(t,\stateone\right)=\frac{1}{N}\suml_{i=1}^N \delta(\stateone-\stateone_i(t)),
\end{equation}
where $\delta$ denotes the Dirac delta.
Similarly, the initial distribution $\fMOone_0$ is obtained as limit for $N\rightarrow\infty$ of the empirical distribution centered at the initial data $\vec{x}_0$.

\begin{lemma}[Single Control  System]\label{Lem:1}
	Consider the optimal control problem in~\eqref{Eq:OCone} of $N$ agents and the optimal control problem ~\eqref{Eq:OCMFone} for the density $\fMOone: [0,T]\times\R^{n}\rightarrow\R$ of agents.
	Let $\fOMone: [0,T]\times\R^{n}\times\R^{n}\rightarrow\R$ be the distribution function, which satisfies the mean-field limit of the first-order the optimality conditions of problem~\eqref{Eq:OCone}.
	The multiplier to the optimality system of~\eqref{Eq:OCMFone} is $\lagrMOone=\lagrMOone(t,\stateone)$.
	
	Then, the solution of the mean-field of the optimality conditions of~\eqref{Eq:OCone} and the solution of the optimality conditions of~\eqref{Eq:OCMFone} are formally identified by:
	\begin{subequations}
	\begin{align}
		\nabla_\stateone\lagrMOone(t,\stateone)&=-		\int \lagrOMone \fOMone_2\left(t,\stateone,\lagrOMone\right)\,\d \lagrOMone,\label{Eq:Lem1Cond2}
	\end{align}
	for all $t\geq0$  and all $\stateone$ in the support of $\fMOone$.
	The function $\fOMone_2$ is the marginal of $\fOMone$:
	\begin{align}
		\fOMone\left(t,\stateone,\lagrOMone\right)&=\fMOone\left(t,\stateone\right)\fOMone_2\left(t,\stateone,\lagrOMone\right).\label{Eq:Lem1Cond1}
	\end{align}
	\end{subequations}
\end{lemma}

Note, the consistency conditions in Lemma~\ref{Lem:1} state a decomposition of the probability density $\fOMone(t,\stateone,\lagrOMone)$ to the probability density $\fMOone(t,\stateone)$ in~\eqref{Eq:Lem1Cond1} and the Lagrange multipliers $\lagrOMone$ and $\lagrMOone(t,\stateone)$ in~\eqref{Eq:Lem1Cond2}.

	The proof of Lemma~\ref{Lem:1} is
	omitted since it is analogous to the proof of Lemma~\ref{Lem:2} presented below.

\subsection{Individual Control  System}\label{Sec:2.2}

We consider the interacting agent system of $N$ agents which reads as follows:
\begin{align}
\begin{split}\label{Eq:OCtwo}
\min\limits_\crtltwo &~\frac{1}{N}\suml^N_{i=1}\int\limits_\tini^\tend \left[\objtwo\left(\crtltwo_i,\tilde{m}\left(\statetwo_i\right)\right) +\frac{\regtwo}{2} \crtltwo_i^2\right]\,\d t\\
\st &~\dot{\statetwo}_i=\frac{1}{N}\suml^N_{j=1} \dyntwo(\statetwo_i,\statetwo_j,\crtltwo_i), \quad\statetwo_i(\tini)=\statetwo_{i,\tini}, \quad i=1,\dots,N
\end{split}
\end{align}
Contrary to Section \ref{Sec:2.1}, each agent $i$ influences the model by its control $\crtltwo_i=\crtltwo_i(t)\in\R^{n_{\crtltwo}}$.
%
%
Formally, we obtain a mean-field optimal control problem as:
\begin{align}
\begin{split}\label{Eq:OCMFtwo}
\min\limits_{\crtltwo}	&~\int\limits_\tini^\tend\int \left[\objtwo\left(\crtltwo,\tilde{m}(\statetwo)\right) +\frac{\regtwo}{2} \crtltwo^2\right]\fMOtwo\,\d x\,\d t\\
\st &~0=\partial_t\fMOtwo+\div_\statetwo \left(\fMOtwo \int \dyntwo(\statetwo,\hat{\statetwo},\crtltwo) \fMOtwo(t,\hat{\statetwo}) \,\d\hat{\statetwo}   \right)\\
&~\fMOtwo(\tini,\statetwo)=\fMOtwo_\tini(\statetwo)
\end{split}
\end{align}
A  difference in~\eqref{Eq:OCMFtwo} compared to the problem in~\eqref{Eq:OCMFone} is that the mean-field control $u$ is additionally dependent on the state space, i.e.  $\crtltwo=\crtltwo(t,\statetwo)$.
If the mean-field density $\fMOtwo$ is chosen to be the empirical measure~\eqref{Eq:EmpirMeas} then
dynamics and cost of the problems~\eqref{Eq:OCtwo} and~\eqref{Eq:OCMFtwo} coincide if we define $u_i(t)=u(t,x_i)$.

As for the problem of a single control, we can derive consistency conditions which connect the optimality conditions of~\eqref{Eq:OCtwo} and \eqref{Eq:OCMFtwo}
\begin{lemma}[Individual Control  System]\label{Lem:2}
		
	Consider the optimal control problem in~\eqref{Eq:OCtwo} of $N$ agents and the optimal control problem~\eqref{Eq:OCMFtwo} for the density $\fMOtwo: [0,T]\times\R^{n}\rightarrow\R$ of agents.
	Let $\fOMtwo: [0,T]\times\R^{n}\times\R^{n}\rightarrow\R$ be the distribution function, which satisfies the mean-field limit of the first-order the optimality conditions to~\eqref{Eq:OCtwo}.
	The multiplier to the  the optimality system of~\eqref{Eq:OCMFtwo} is $\lagrMOtwo=\lagrMOtwo(t,\statetwo)$.
	
	Then, the solution of the mean-field of the optimality conditions of~\eqref{Eq:OCtwo} and the solution of the optimality conditions of~\eqref{Eq:OCMFtwo} can be formally identified by:
		\begin{subequations}\label{Eq:Lem2Cond12}
			\begin{equation}
			    \nabla_\statetwo\lagrMOtwo(t,\statetwo)=	-\int \lagrOMtwo \fOMtwo_2\left(t,\statetwo,\lagrOMtwo\right)\,\d \lagrOMtwo,\label{Eq:Lem2Cond2}
			\end{equation}
			for all $t\geq0$ and all $\statetwo$ in the support of $\fMOtwo$.
			The function $\fOMtwo_2$ is the marginal of $\fOMtwo$:
			
	\begin{equation}
	\fOMtwo\left(t,\statetwo,\lagrOMtwo\right)=\fMOtwo\left(t,\statetwo\right)\fOMtwo_2\left(t,\statetwo,\lagrOMtwo\right).\label{Eq:Lem2Cond1}
	\end{equation}
		\end{subequations}

\end{lemma}
	The proof of Lemma~\ref{Lem:2} is carried out in Section~\ref{SubSec:Lem2}.
	In particular, the   optimality conditions of~\eqref{Eq:OCtwo} may be found in (\ref{Eq:OptOMtwo1}-\ref{Eq:OptOMtwo3}) and its mean-field limit in (\ref{Eq:OM1two}-\ref{Eq:OM2two}).
	The optimality conditions of~\eqref{Eq:OCMFtwo} are given in (\ref{Eq:MO3two}-\ref{Eq:MO2two}).

\begin{corollary}[Parameterized Problems]\label{Cor:Parameter}
	Consider the optimal control problem in~\eqref{Eq:OCtwo} with the parameterized objective  $J=J\left(\crtltwo_i,\tilde{m}\left(\statetwo_i\right);p\right)$.
	Then Lemma~\ref{Lem:2} holds also for the parameterized objective.
\end{corollary}

Before proving Lemma \ref{Lem:2},  we  address  an aspect  related to the usage of $L^2$ calculus for the formal computations of the optimality conditions of~\eqref{Eq:OCMFtwo}.
In the proof, Gateaux derivatives of the Lagrangian are computed, see~\eqref{Eq:MOtwo}.
In particular, the derivative with respect to the probability density $\fMOtwo$ is computed.
Probability densities are nonegative and their integral is one.
A consistent derivative with respect to such a function conserves these properties also with the variation, e.g. in Wasserstein calculus.
This means a suitable variation $\eta$ of the probability density $\fMOtwo$ satisfies:
\begin{equation}\label{Eq:Variation}
	\fMOtwo(t,x)+\eta(x)\geq0 \quad\text{ and }\quad\int_{\R^n} \left(\fMOtwo(t,x)+\eta(x)\right)\,\d x=1,
\end{equation}
which is not the case in $L^2$ calculus.
However, this relation is recovered by~\eqref{Eq:Lem2Cond12}.
Assume in the following paragraph that $n=1$.
The Lagrangian of the problem~\eqref{Eq:OCMFtwo} contains the scalar product of the evolution equation of $\fMOtwo$ with the multiplier $\lagrMOtwo$:
\begin{align*}
	\mathcal{L}^{\mathrm{MO}}\left(\fMOtwo,\crtltwo,\lagrMOtwo\right)=&\int\limits_\tini^\tend\int \left[\objtwo\left(\crtltwo,\tilde{m}(\statetwo)\right) +\frac{\regtwo}{2} \crtltwo^2\right]\fMOtwo\,\d x\,\d t\\
	& +\left\langle  \partial_t\fMOtwo+\div_\statetwo \left(\fMOtwo \int \dyntwo(\statetwo,\hat{\statetwo},\crtltwo) \fMOtwo(t,\hat{\statetwo}) \,\d\hat{\statetwo}   \right),   \lagrMOtwo\right\rangle.
\end{align*}
If one uses now instead of the standard $L^2$ scalar product the following scalar product:
\begin{equation*}
\left\langle \fMOtwo,\lagrMOtwo\right\rangle:=\int\,\fMOtwo(t,\statetwo) \partial_\statetwo\lagrMOtwo(t,\statetwo)\d\statetwo,
\end{equation*}
we have that  $\partial_\statetwo \lagrMOtwo$ is  a consistent variation of $\fMOtwo$ for compactly supported $\lagrMOtwo$ since:
$$\int_\R\partial_\statetwo \lagrMOtwo(t,\statetwo) \,\d\statetwo=0$$
Hence the suitable test function in~\eqref{Eq:Variation} is $\eta(x)=\partial_\statetwo \lagrMOtwo(t,\statetwo)$.

\subsection{Proof of Lemma~\ref{Lem:2}} \label{SubSec:Lem2}

We refer to \cite{Herty2019a} for a detailed discussion.

Notation. A function with
an hat is evaluated in space or multiplier of the hat variable, e.g. $\fOMtwohat=\fOMtwo(t,\hat{\statetwo},\hat{\lagrOMtwo})$ and $\fOMtwohat=\fOMtwo(t,\hat{\statetwo})$.

\paragraph{First Optimize then Mean-Field Limit}

Under regularity assumptions on the cost and the dynamics $\dynone$,  Pontryagin's maximum principle provides optimality conditions to~\eqref{Eq:OCtwo}.
The first-order optimality conditions are composed of the state dynamics and the dynamics of the Lagrange multipliers $\left(\lagrOMtwo_i\right)_{i=1}^N\in\R^{Nn}$:
\begin{subequations}\label{Eq:OptOMtwo}
\begin{align}
\dot{\statetwo}_i=&  \frac{1}{N}\suml^N_{j=1} \dyntwo(\statetwo_i,\statetwo_j,\crtltwo_i)    ,\label{Eq:OptOMtwo1}\\
\begin{split}
\dot{\lagrOMtwo}_i =&-\D_\statetwo\tilde{m}(\statetwo_i)^\top\nabla_m \objtwo\left(\crtltwo_i,\tilde{m}\left(\statetwo_i\right)\right)  \\&- \frac{1}{N}\suml_{j=1}^N\left[ \D_1 \dyntwo\left(\statetwo_i,\statetwo_j,\crtltwo_i\right)^\top\lagrOMtwo_i+\D_2 \dyntwo\left(\statetwo_j,\statetwo_i,\crtltwo_j\right)^\top\lagrOMtwo_j\right],  \label{Eq:OptOMtwo2}
\end{split}
\end{align}
with $\statetwo_i(\tini)=\statetwo_{i,\tini}$ and $\lagrOMtwo_i(\tend)=0$ for $i=1,\dots,N$ and in addition to that, the control is determined by:
\begin{equation}\label{Eq:OptOMtwo3}
 0= \nabla_\crtltwo\objtwo\left(\crtltwo_i,\tilde{m}\left(\statetwo_i\right)\right) +\regtwo \crtltwo_i+ \frac{1}{N}   \suml^N_{j=1} \D_\crtltwo\dyntwo(\statetwo_i,\statetwo_j,\crtltwo_i)^\top\lagrOMtwo_i. \end{equation}
\end{subequations}

In this paragraph, we derive evolution equation of the probability density $\fOMtwo=\fOMtwo(t,\statetwo,\lagrOMtwo)$.
To derive the mean field limit, we assume there exists an $u: [0,\tend]\times\R^n\rightarrow\R^{n_u}$ such that: $$\crtltwo(t,\statetwo_i(t))=\crtltwo_i(t),$$
 for $t\geq 0$ and all $i=1,\dots,N$.
The mean-field equation  related to the many particle limit of the dynamical system (\ref{Eq:OptOMtwo1}-\ref{Eq:OptOMtwo2})  is:
\begin{subequations}\label{Eq:OMtwo}
\begin{align}
\begin{split}\label{Eq:OM1two}
0=&\partial_t\fOMtwo+\div_\statetwo \left(\fOMtwo \int \dyntwo(\statetwo,\hat{\statetwo},\crtltwo) \fOMtwohat \,\d\hat{\statetwo}\,\d\hat{\lagrOMtwo}   \right)\\
& -\div_\lagrOMtwo \left(  \fOMtwo \left[\int \left[ \D_1\dyntwo(\statetwo,\hat{\statetwo},\crtltwo)^\top\lagrOMtwo + \D_2\dyntwo(\hat{\statetwo},\statetwo,\hat{\crtltwo})^\top \hat{\lagrOMtwo}\right] \fOMtwohat \,\d\hat{\statetwo}\,\d\hat{\lagrOMtwo} \right.\right.\\&\left.~~~~~~~~~~~~~~~~~~~~~~~~+ \D_\statetwo\tilde{m}(\statetwo)^\top\nabla_m\objtwo(\crtltwo,\tilde{m}(\statetwo))\bigg] \right),
\end{split}
\end{align}
with the initial condition $\fOMtwo(\tini,\statetwo,\lagrOMtwo)=\fOMtwo_\tini(\statetwo,\lagrOMtwo)$ for all $(\statetwo,\lagrOMtwo)$.
The mean-field limit to~\eqref{Eq:OptOMtwo3} is:
\begin{equation}\label{Eq:OM2two}
0= \nabla_\crtltwo\objtwo\left(\crtltwo,m_{\fOMtwo}(t)\right) +\regtwo \crtltwo+ \int \D_\crtltwo\dyntwo(\statetwo,\hat{\statetwo},\crtltwo)^\top\lagrOMtwo \fOMtwohat\,\d\hat{\statetwo}\,\d\hat{\lagrOMtwo},
\end{equation}
\end{subequations}
where $m_{\fOMtwo}(t)=\int \tilde{m}(\statetwo) \fOMtwo\,\d\statetwo\,\d\lagrOMtwo$.

\paragraph{First Mean-Field Limit then Optimize}
The formal first-order optimality conditions of~\eqref{Eq:OCMFtwo} in the $L^2$-sense are given by:
\begin{subequations}\label{Eq:MOtwo}
\begin{align}
0=& \partial_t\fMOtwo+\div_\statetwo\left(\fMOtwo \int \dyntwo(\statetwo,\hat{\statetwo},\crtltwo) \fMOtwohat \,\d\hat{\statetwo}   \right), \label{Eq:MO3two}\\
\begin{split}
0=& \objtwo\left(\crtltwo , \tilde{m}(\statetwo)\right)+\frac{\regtwo}{2}\crtltwo^2-\partial_t \lagrMOtwo\\&- \int \left( \dyntwo\left(\statetwo,\hat{\statetwo},\crtltwo\right)^\top \nabla_\statetwo \lagrMOtwo+ \dyntwo\left(\hat{\statetwo},\statetwo,\hat{\crtltwo}\right)^\top\nabla_{\hat{\statetwo}} \hat{\lagrMOtwo}\right) \fMOtwohat\,\d \hat{\statetwo},\label{Eq:MO1two}\end{split}\\
0=& \nabla_\crtltwo \objtwo \left(\crtltwo, \tilde{m}\left(\statetwo\right)\right)+\regtwo\crtltwo- \int \D_\crtltwo \dyntwo\left(\statetwo,\hat{\statetwo},\crtltwo\right)^\top \nabla_\statetwo\lagrMOtwo\fMOtwohat \,\d \hat{\statetwo},    \label{Eq:MO2two}
\end{align}
\end{subequations}
with the initial value $\fMOtwo(\tini,\statetwo)=\fMOtwo_\tini(\statetwo)$ and the terminal condition $\lagrMOtwo(\tend,\statetwo)=0$ for all $\statetwo$.

\paragraph{The Relation between the Approaches}
In this paragraph, we connect the two approaches and derive the consistency conditions \eqref{Eq:Lem2Cond12}.
We may assume that there exists a decomposition such that:
\begin{equation*}
\fOMtwo\left(t,\statetwo,\lagrOMtwo\right)=\fOMtwo_1\left(t,\statetwo\right)\fOMtwo_2\left(t,\statetwo,\lagrOMtwo\right),
\end{equation*}
where for the conditional probability density   $\int\fOMone_2\left(t,\stateone,\lagrOMone\right)\,\d \lagrOMone=1$ holds.
Upon multiplication of~\eqref{Eq:OM2two} by $\fOMtwo_2$, integration with respect to $\lagrOMtwo$ yields:
\begin{align}
\begin{split}\label{Eq:ObjProof}
0=& \nabla_\crtltwo\objtwo\left(\crtltwo,m_{\fOMtwo_1\fOMtwo_2}(t)\right)\int \fOMtwo_2\,\d\lagrOMtwo +\regtwo \crtltwo \int \fOMtwo_2\,\d\lagrOMtwo\\ &+  \int \D_\crtltwo\dyntwo(\statetwo,\hat{\statetwo},\crtltwo)^\top \fOMtwohat_1\fOMtwohat_2\,\d\hat{\statetwo}\,\d\hat{\lagrOMtwo} \int\lagrOMtwo \fOMtwo_2\,\d\lagrOMtwo,
\end{split}
\end{align}
Hence, if:
\begin{subequations}\label{Eq:Ass23two}
\begin{align}
\fOMtwo_1(t,\statetwo)&=\fMOtwo(t,\statetwo),\label{Eq:Ass2two}\\
\int \lagrOMtwo \fOMtwo_2\,\d\lagrOMone&=-\nabla_\statetwo \lagrMOtwo(t,\statetwo)\label{Eq:Ass3two}
\end{align}
\end{subequations}
then Equation~\eqref{Eq:ObjProof} coincides with~\eqref{Eq:MO2two}.

Using the assumptions in~\eqref{Eq:Ass23two}, we obtain that after integration of the Equation~\eqref{Eq:OM1two} with respect to $\lagrOMtwo$:
\begin{align*}
0=\partial_t\int\fOMtwo_1\fOMtwo_2\,\d\lagrOMtwo +\div_\statetwo\int\fOMtwo_1 \fOMtwo_2  \dyntwo(\statetwo,\hat{\statetwo},\crtltwo) \fOMtwohat_1\fOMtwohat_2 \,\d\hat{\statetwo}\,\d\hat{\lagrOMtwo}\,\d\lagrOMtwo,
\end{align*}
which is equivalent to~\eqref{Eq:MO3two}.

We continue by multiplication of~\eqref{Eq:OM1two} by $\lagrOMtwo$ and insertion of~\eqref{Eq:Ass3two}, then integration by parts yields:
\begin{align*}
\begin{split}
0=&-\partial_t\left(\fOMtwo_1 \nabla_\statetwo\lagrMOtwo \right)-\div_\statetwo \left(\fOMtwo_1 \nabla_\statetwo\lagrMOtwo   \int \dyntwo(\statetwo,\hat{\statetwo},\crtltwo)^\top \fOMtwohat_1 \,\d\hat{\statetwo}\right)\\
& +\fOMtwo_1\iint   \left[ \D_1\dyntwo(\statetwo,\hat{\statetwo},\crtltwo)^\top \lagrOMtwo + \D_2\dyntwo(\hat{\statetwo},\statetwo,\hat{\crtltwo})^\top \hat{\lagrOMtwo} \right] \fOMtwohat_1\fOMtwohat_2\fOMtwo_2 \,\d\hat{\statetwo}\,\d\hat{\lagrOMtwo} \,\d \lagrOMtwo\\
& +\fOMtwo_1\D_\statetwo\tilde{m}(\statetwo)^\top\nabla_m \objtwo\left(\crtltwo,\tilde{m}(\statetwo)\right).
\end{split}
\end{align*}
This equation is equivalent to~\eqref{Eq:MO1two} when computing the gradient with respect to $\nabla_\statetwo$:
\begin{align*}
\begin{split}
0=& - \fOMtwo_1\left\{ \partial_t  \nabla_\statetwo\lagrMOtwo + \nabla_\statetwo^2\lagrMOtwo  \int\dyntwo(\statetwo,\hat{\statetwo},\crtltwo) \fOMtwohat_1 \,\d\hat{\statetwo} +\int \D_1 \dyntwo(\statetwo,\hat{\statetwo},\crtltwo)^\top \nabla_\statetwo\lagrMOtwo  \fOMtwohat_1 \,\d\hat{\statetwo}\right.\\
& \left.+\int  \D_2 \dyntwo(\hat{\statetwo},\statetwo,\hat{\crtltwo})^\top \nabla_{\hat{\statetwo}}\hat{\lagrMOtwo} \fOMtwohat_1 \,\d\hat{\statetwo} -\D_\statetwo\tilde{m}(\statetwo)^\top\nabla_m \objtwo\left(\crtltwo,\tilde{m}(\statetwo)\right)\right\}.	
\end{split}
\end{align*}
This completes the proof. 
\section{Stackelberg Game}\label{Sec:3}

We introduce a linear-quadratic Stackelberg game, consisting of the leader and $N$  followers.
The interchangeability of mean field limit and optimization is discussed using the results from Section~\ref{Sec:2}.

The Stackelberg game to be discussed reads as follows:
\begin{align}
\begin{split}\label{Eq:Stack}
\min\limits_{\crtllead} & ~\int\limits_\tini^\tend \left[\objlead(\crtllead,m(\vec{\statefoll})) +\frac{\reglead}{2} \crtllead^2\right]\,\d t\\
\st 
 &~\min\limits_{\vec{\crtlfoll},\vec{\statefoll}}~ \frac{1}{N}\suml_{i=1}^N\int\limits_{\tini}^\tend \left[\objfoll( \tilde{m}(\statefoll_i);\crtllead)+\frac{\regfoll}{2} \crtlfoll_i^2\right]\,\d t\\
 & ~~~~ \st ~~\dot{\statefoll}_i=\frac{1}{N}\suml_{j=1}^N \dynfoll\left(\statefoll_i,\statefoll_j\right)\left(\statefoll_j-\statefoll_i\right) + \crtlfoll_i, \quad \statefoll_i(\tini)=\statefoll_{i,0}, \quad i=1,\dots,N,
\end{split}
\end{align}
where the leader minimizes the function $\objlead: \R^{n_L}\times\R^{n_F}\rightarrow\R$ regularized
by  a quadratic term of its control $\crtllead=\crtllead(t)\in\R^{n_L}$ with $\reglead>0$.
The structure of the followers' control problem is similar: Every follower $i\in\{1,\dots,N\}$ aims to minimize $\objfoll: \R^{n_F}\times\R^{n_L}\rightarrow\R$ regularized by its quadratic control $\crtlfoll_i=\crtlfoll_i(t)\in\R^{n_F}$ with regularization parameter $\regfoll>0$.
The value $m=m(\vec{\statefoll })$ is considered be a vector of  moments of the states  $m: \R^{Nn_L}\rightarrow\R^{n_L}$.

The structure of the followers' problem  is a potential game~\cite{Monderer1996}, i.e. here the followers' states and controls are not coupled in their objective functions.
Therefore they can be summed up which yields cooperation between the followers.

The followers' problem is governed by an ordinary differential equation of every follower's state $\statefoll_i=\statefoll_i(t)\in\R^{n_F}$ which  couples to the other followers via the interaction kernel $\dynfoll: \R^{n_F}\times\R^{n_F}\rightarrow\R$.
The concatenation of the followers' states and controls is denoted by $\vec{\statefoll}\in\R^{Nn_F}$ and $\vec{\crtlfoll}\in\R^{Nn_F}$.

Unlike the optimal control problems in Section~\ref{Sec:2}, Stackelberg games have multiple levels of optimization.
Therefore the optimality conditions have to be derived in a systematic order.
This results in three different possibilities illustrated in Figure~\ref{Fig:OrderOfOptMF}.
\begin{theorem}\label{Th:MainRes}
Consider the Stackelberg game in~\eqref{Eq:Stack} of a single leader and $N$ followers.

Denote by $\fOOM:[0,T]\times\R^{n_F}\times\R^{n_F}\times\R^{2n_F}\rightarrow\R$
the probability density of the mean-field limit of the followers after optimization of both, the leader and the followers.

Furthermore, let $\fOMO:[0,T]\times\R^{n_F}\times\R^{n_F}\rightarrow\R$
 denote the probability density of the mean field limit  after optimization of the leader and  $\fMOO:[0,T]\times\R^{n_F}\rightarrow\R$
 the probability density of the followers prior to their optimization.

The function $\Theta:[0,T]\times\R^{n_F}\times\R^{n_F}\rightarrow\R$ is the multiplier to $\fOMO$ and the function $\Psi: [0,T]\times\R^{n_F}\rightarrow\R $ is the multiplier to $\fMOO$.

If the condition:
\begin{align}\label{Eq:VarZeroCond}
\int \lagrOMfoll^2 \fOMO_2(t,\statefoll,\lagrOMfoll) \,\d\lagrOMfoll=\left(\int \lagrOMfoll \fOMO_2(t,\statefoll,\lagrOMfoll) \,\d\lagrOMfoll\right)^2,
\end{align}
holds for $t\geq0$ and all $\statefoll\in\R^{n_F}$, then the three optimality systems are equivalent in the mean-field limit and we have  following relations:
\begin{subequations}
	\begin{align}
\fOOM(t,\statefoll,\psi,\theta)&=\fOMO(t,\statefoll,\psi)\fOOM_2(t,\statefoll,\psi,\theta),\\
\nabla_{\statefoll,\psi}\Theta(t,\statefoll,\psi)&=\begin{bmatrix}
\nabla_{\statefoll}\Theta(t,\statefoll,\psi)\\\nabla_{\psi}\Theta(t,\statefoll,\psi)
\end{bmatrix}=-\int\theta \fOOM_2(t,\statefoll,\psi,\theta)\,\d\theta,
\end{align}
for all $t\geq0$ and $\statefoll$ in the support of  $\fOMO$ and:
\begin{align}
\fOMO(t,\statefoll,\lagrOMfoll)&=\fMOO(t,\statefoll)\fOMO_2(t,\statefoll,\lagrOMfoll),\\
\nabla_\statefoll \lagrMOfoll(t,\statefoll)&=-\int \lagrOMfoll\fOMO_2(t,\statefoll,\lagrOMfoll)\,\d \lagrOMfoll,\label{Eq:Consistency4}
\end{align}
\end{subequations}
for all $t\geq0$ and all $\statefoll$  in the support of $\fMOO$.
\end{theorem}
Before proving this statement in Section~\ref{Sec:3.1}, we provide additional interpretation.

The variance of a random variable $X$ with the realization $x$ and the probability density $\rho$ is defined as $	\mathrm{Var}(X)=\mathbb{E}\left[X^2\right]-\mathbb{E}\left[X\right]^2$,
where $\mathbb{E}\left[X\right]$ denotes the expectation of $X$, which is defined as
$\mathbb{E}\left[X\right]=\int x \rho(x) \,\d x$.

With this,  condition~\eqref{Eq:VarZeroCond} is equivalent to requiring that the variance  of a random variable $Y$ with the probability density $y\mapsto\fOMO_2(t,\statefoll,y)$ is zero for all $(t,\statefoll)$, i.e. with the realization $y$ of $Y$ we have:
\begin{equation*}
	\mathrm{Var}\left(Y\right)=\int y^2 \fOMO_2(t,\statefoll,y) \,\d y-\left(\int y \fOMO_2(t,\statefoll,y) \,\d y\right)^2=0.
\end{equation*}
That is e.g. the case if the probability density $\fOMO_2(t,\statefoll,y)$ coincides with the empirical measure in $y$ concentrated on $\mathbb{E}[Y]$, i.e. $\fOMO_2(t,\statefoll,y)=\delta\left(y-y(t,\statefoll)\right)$.

In the proof for Theorem~\ref{Th:MainRes}, it is  shown that the optimal follower control is $\crtlfoll(t,\statefoll)=\frac{1}{\regfoll}\nabla_\statefoll\Psi$ which is in fact then:
\begin{equation*}
	\crtlfoll(t,\statefoll)=-\frac{1}{\regfoll}\mathbb{E}\left[Y\right].
\end{equation*}

\subsection{Proof of Theorem~\ref{Th:MainRes}} \label{Sec:3.1}

\begin{figure}
	\tikzstyle{block1} = [rectangle, draw=black, text centered,   minimum height=5em, text width=12em]
	\tikzstyle{block2} = [rectangle, draw=black, text centered,   minimum height=5em, text width=18em]
	\tikzstyle{line} = [draw,line width=0.05cm ,-latex']
	\tikzstyle{lined} = [dashed, line width=0.05cm, -latex']
	\tikzstyle{lined2} = [dashed, line width=0.05cm]
	\begin{center}
		\resizebox{\textwidth}{!}{
			\begin{tikzpicture}[node distance = 4cm, auto]
			\node [block1] (0) {Discrete Stackelberg\\
				Two level problem\\
				ODE in $\statefoll_i$
				\eqref{Eq:Stack}
			};
			\node [below of=0,node distance = 4cm] (ABla){};
			\node [below of=0,node distance = 3.3cm] (ABla2) {\textbf{Step 2}\quad\quad\quad\quad\quad\quad~\quad};			
			\node [block2, left of=ABla,node distance=15em] (A1) 
			{ Single level problem\\
				ODE in $\dot{\zeta}_i=[\dot{\statefoll}_i,\dot{\lagrOMfoll}_i]$	~\eqref{Eq:Summary2}
				
			}; 
			\node [block1, right of=ABla,node distance=15em] (A2)
			{ Mean Field Stackelberg\\
				Two level problem\\
				PDE in $\fMOO(t,\statefoll)$
				\eqref{Eq:Summary1}
				
			}; 
			\node [below of=A1] (BBla) {};
			\node [block1, left of=BBla,node distance=7.5em] (B1) 
			{Optimality system\\
				ODE in $\dot{\zeta}_i$, $\dot{\theta}_i$~\eqref{Eq:Summary5}
			}; 
			\node [block1, right of=BBla,node distance=7.5em] (B2) 
			{Singel level problem\\
				PDE in $\fOMO(t,\statefoll,\lagrOMfoll)$
				\eqref{Eq:Summary4}
			}; 
			\node [block1, below of=A2] (B3) 
			{Single level problem\\
				PDE in $\fMOO(t,\statefoll)$, $\Psi(t,\statefoll)$ 
				\eqref{Eq:Summary3}
			}; 
			\node [block1, below of=B1] (C1) 
			{ Optimality system\\
				PDE in $\fOOM(t,\zeta,\theta)$
				\eqref{Eq:Summary8}
			}; 
			\node [block1, below of=B2] (C2) 
			{Optimality system\\
				PDE in $\fOMO(t,\statefoll,\lagrOMfoll)$,  $\Theta(t,\statefoll,\psi)$
				\eqref{Eq:Summary6}
			}; 
			\node [block1, below of=B3] (C3) 
			{Optimality system\\
				PDE in $\fMOO(t,\statefoll)$, $\Psi(t,\statefoll)$,  $\lagrMOlead_1(t,\statefoll)$, $\lagrMOlead_2(t,\statefoll)$
				\eqref{Eq:FullOptMOO}
			}; 

			\path [line,line width=0.05cm] (0) -- (A1) node [right, pos=0.5 ]{~OPT} ;
			\draw [lined,black] (0.south west) --  (A1.north) node [left, pos=0.4 ]{\textbf{Step 1} \quad};
			
			\node[below of=C2, node distance=0.9cm, black] (Dummy1) {};
			\node[right of=Dummy1, node distance=2.5cm, black] (Dummy2) {};
			
			\node[draw=black, dashed, line width=0.05cm, fit=(A1)(B1) (B2) (C1) (C2) (Dummy2)](Step2line) {};
			
			\node[above of=B2, node distance=1.2cm, black] (Step3) {\quad\quad\quad\quad\quad\quad\quad \textbf{Step 3} };
			\node[draw=black, dashed, line width=0.05cm, fit= (B2)(Step3)](Step3line) {};
			
			\draw [lined,black] (B2.east) --  (B3.west) node [below, pos=0.5 ]{\begin{minipage}{1.5cm}\begin{center}
				\textbf{Step 6}
				\end{center}
				\end{minipage} };
			\draw [lined,black] (B3.west) --  (B2.east) node [below, pos=0.5 ]{};
			
			\draw [lined,black] (A2.south  west) --  (B3.north west) node [right, pos=0.5 ]{ \textbf{Step 5}};
			\draw [lined,black] (0.south east) --  (A2.north) node [right, pos=0.5 ]{\quad \textbf{Step 4}};
			
			\path [line] (0) -- (A2) node [left, pos=0.5 ]{MF ~} ;
			\path [line] (A1) -- (B1) node [left, pos=0.5 ]{OPT} ;
			\path [line] (A1) -- (B2) node [right, pos=0.5 ]{MF} ;
			\path [line] (A2) -- (B3) node [right, pos=0.5 ]{OPT} ;
			\path [line] (B1) -- (C1) node [right, pos=0.5 ]{MF} ;
			\path [line] (B2) -- (C2) node [right, pos=0.5 ]{OPT} ;
			\path [line] (B3) -- (C3) node [right, pos=0.5 ]{OPT} ;
			\end{tikzpicture}
		}
	\end{center}
	\caption{This graph illustrates the proof of Theorem~\ref{Th:MainRes}. The solid lines indicate the order of optimization \textrm{(OPT)} and mean field limit \textrm{(MF)}. The dashed lines refer to the steps in the proof.}\label{Fig:Graph}
\end{figure}
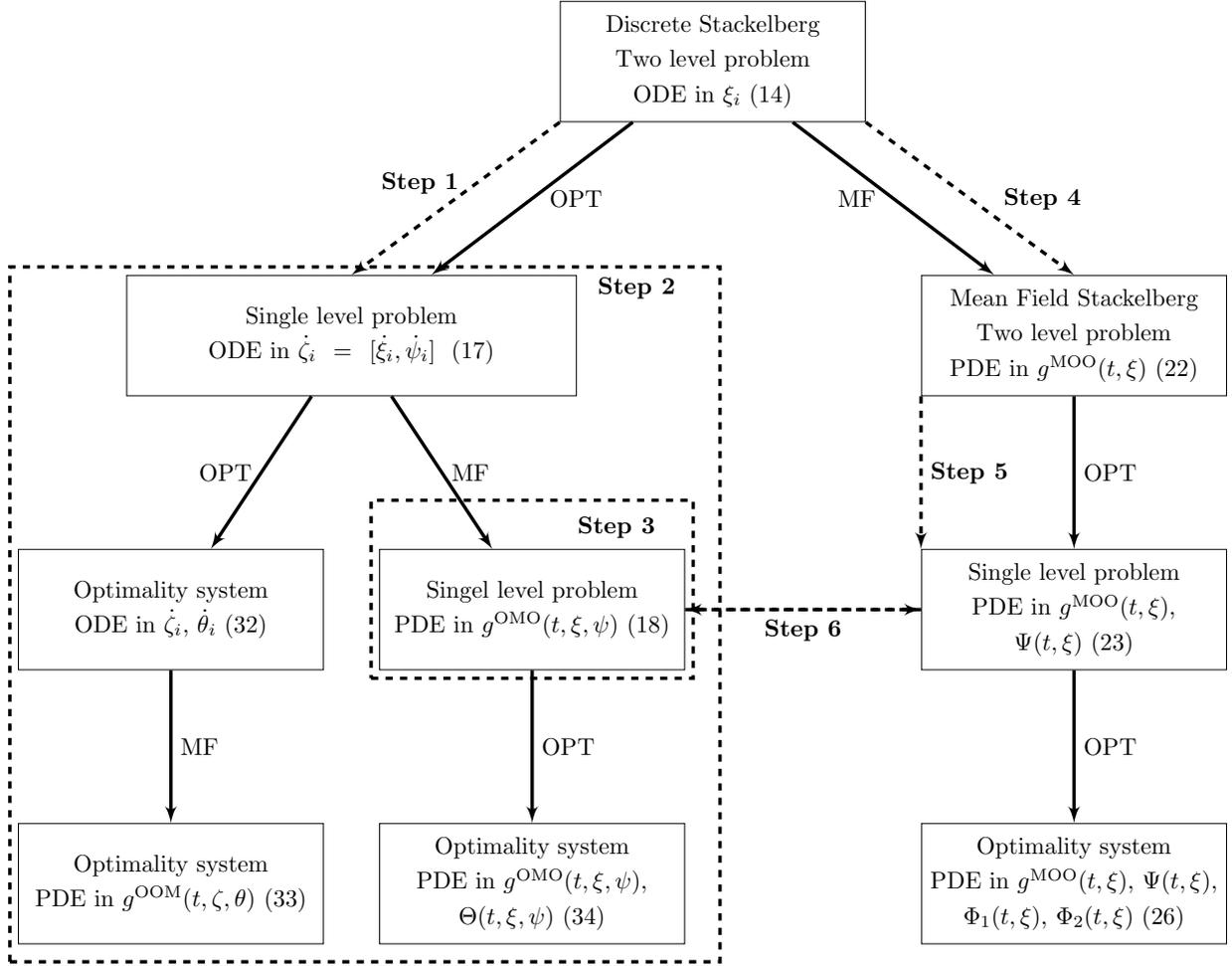

The proof consists of six steps, in which we follow the graph of Figure~\ref{Fig:Graph}.
For notation convenience, arguments of functions are omitted.

	\emph{Step 1.} We begin the analysis on the follower level of the game in Equation~\eqref{Eq:Stack}.
	Under regularity assumptions, the followers' optimization problem  allows using Pontryagin's maximum principle.
	With this, the optimal response of the followers to the leader  can be characterized by a coupled system of ordinary differential equations (ODE).
	It consists of the state dynamic ${\statefoll}_i$ and the dynamic of the dual ${\lagrOMfoll}_i$ for every follower $i=1,\dots,N$.
	The optimal controls $\crtlfoll_i$ are available explicitly and  substituted in the ODE system.
	Replacing the optimal control problem of the follower by the ODE system yields the leader's optimal control problem:
	\begin{align}\label{Eq:Summary2}
	\begin{split}
	\min\limits_{\crtllead} & ~\int\limits_\tini^\tend 	\left[\objlead(\crtllead,m(\vec{\statefoll})) +\frac{\reglead}{2} \crtllead^2\right]\,\d t\\
	\st 
	&~\dot{\statelead}_i=\begin{bmatrix}
	\dot{\statefoll}_i\\\dot{\lagrOMfoll}_i
	\end{bmatrix}=\frac{1}{N}\suml_{j=1}^N						G(\statelead_i,\statelead_j,v),\quad i=1,\dots,N\\
	& ~ \statefoll_i(0)=\statefoll_{i,0},\quad \lagrOMfoll_i(T)=0, \quad i=1,\dots,N,\\
	\end{split}
	\end{align}
	where the dynamic is composed as $G(\statelead_i,\statelead_j,v)=\begin{bmatrix}
		G_1(\statelead_i,\statelead_j,v)\\G_2(\statelead_i,\statelead_j,v)
	\end{bmatrix}$ for:
	\begin{align*}
        G_1(\statelead_i,\statelead_j,v)=&\dynfoll(\statefoll_i,\statefoll_j)(\statefoll_j-\statefoll_i) -\frac{1}{\regfoll} \lagrOMfoll_i,\\
		G_2(\statelead_i,\statelead_j,v)=&-\D_\statefoll\tilde{m}(\statefoll_i)^\top \nabla_m\objfoll(\tilde{m}(\statefoll_i);v)  -\D_{\statefoll_i}\left[ \dynfoll(\statefoll_i,\statefoll_j)(\statefoll_j-\statefoll_i)\right]^\top \lagrOMfoll_i\\
		&-\D_{\statefoll_j}\left[ \dynfoll(\statefoll_j,\statefoll_i)(\statefoll_i-\statefoll_j)\right]^\top \lagrOMfoll_j.
	\end{align*}
	\emph{Step 2.} Starting from the optimization problem in~\eqref{Eq:Summary2}, we now show the interchangeability of optimization and mean field limit for this problem and give the necessary conditions.
	
	We apply Lemma~\ref{Lem:1} to Problem~\eqref{Eq:Summary2} and its formal mean-field limit for the probability density $\fOMO=\fOMO(t,\zeta)$:
	\begin{align}\label{Eq:Summary4}
	\begin{split}
	\min\limits_{\crtllead} & ~\int\limits_\tini^\tend \left[\objlead(\crtllead,m_{\fOMO}(t) )+\frac{\reglead}{2} \crtllead^2\right]\,\d t\\
	\st &~0=\partial_t\fOMO+\div_\statelead \left(\fOMO \int G(\statelead,\hat{\statelead},v) \fOMOhat \,\d\hat{\statelead}   \right)\\
	&~\fOMO(\tini,\statelead)=\fOMO_\tini(\statelead)\\
	&~\fOMO(\tend,\statelead)=\fOMO_\tend(\statelead).\\
	\end{split}
	\end{align}
	Hence, provided that we have on the support of $\fOMO$:
	\begin{equation*}
			\nabla_\statelead\Theta(t,\statelead)=-\int\theta \fOOM_2(t,\statelead,\theta)\,\d\theta,
	\end{equation*}
	for the multiplier $\Theta$ to $\fOOM$, we obtain that the optimality systems to \eqref{Eq:Summary2} and \eqref{Eq:Summary4} coincides in the mean-field limit.
	The probability density $\fOOM$ corresponding to the formal first order optimality system of problem~\eqref{Eq:Summary2} is given in terms of $\fOMO$ and $\fOOM_2$ according to:
	\begin{equation*}
		\fOOM(t,\statelead,\theta)=\fOMO(t,\statelead)\fOOM_2(t,\statelead,\theta).
	\end{equation*}

	\emph{Step 3.}
	Now, the formal mean-field optimal control problem~\eqref{Eq:Summary4} is reformulated. 
	We denote by $\fMOO$ the probability density fulfilling the state equation in~\eqref{Eq:Summary1}.
	
	We decompose  $\fOMO$ as follows:
			\begin{equation*}
				\fOMO(t,\statefoll,\lagrOMfoll)=\hMOO(t,\statefoll)\fOMO_2(t,\statefoll,\lagrOMfoll),\label{Eq:Ansatz1}
			\end{equation*}
			where $\int\fOMO_2(t,\statefoll,\lagrOMfoll)\,\d \lagrOMfoll=1$.
			Furthermore, we denote the expected value by:
		\begin{align*}
		\rlam(t,\statefoll):=&\int \lagrOMfoll\fOMO_2(t,\statefoll,\lagrOMfoll)\,\d \lagrOMfoll.
		\end{align*}

	\emph{a)} Inserting this ansatz in the objective in~\eqref{Eq:Summary4} yields:
	\begin{align}
		\int_0^T \left[\objlead(\crtllead,m_{\hMOO\fOMO_2}(t))+\frac{\reglead}{2}\crtllead^2\right]\,\d t=		\int_0^T \left[\objlead(\crtllead,m_{\hMOO}(t))+\frac{\reglead}{2}\crtllead^2\right]\,\d t.\label{Eq:Step3a}
	\end{align}

		\emph{b)} For the dynamic we have:
	\begin{align*}
	0=\partial_t \hMOO &+ \int\div_\statefoll \left( \hMOO\fOMO_2 \int G_1(\statelead,\hat{\statelead}, \crtllead) \hMOOhat\fOMOhat_2\,\d\hat{\statelead}\right) \,\d\lagrOMfoll\\
	&+ \int\div_\lagrOMfoll \left( \hMOO\fOMO_2 \int G_2(\statelead,\hat{\statelead}, \crtllead) \hMOOhat\fOMOhat_2\,\d\hat{\statelead}\right) \,\d\lagrOMfoll,
	\end{align*}
	and using the definition of $G_1$ we have:
	\begin{align*}
		0=&\partial_t \hMOO + \div_\statefoll \left( \hMOO \left[ \int \dynfoll(\statefoll,\hat{\statefoll})(\hat{\statefoll}-\statefoll) \hMOOhat\,\d\hat{\statefoll} 
	-\frac{1}{\regfoll} \int \lagrOMfoll \fOMO_2 \,\d \lagrOMfoll
	\right] \right),\\
	=&\partial_t \hMOO + \div_\statefoll \left( \hMOO\int\fOMO_2  \left[\dynfoll(\statefoll,\hat{\statefoll})(\hat{\statefoll}-\statefoll)-\frac{1}{\regfoll}\lagrOMfoll \right] \hMOOhat\fOMOhat_2\,\d\hat{\lagrOMfoll}\,\d\hat{\statefoll} \,\d\lagrOMfoll\right).
	\end{align*}
	We simplify integrals and have the expression by the definition of $\rlam$:
	\begin{align}\label{Eq:Step3b}
	0=\partial_t \hMOO + \div_\statefoll \left( \hMOO \left[ \int \dynfoll(\statefoll,\hat{\statefoll})(\hat{\statefoll}-\statefoll) \hMOOhat\,\d\hat{\statefoll} 
	-\frac{1}{\regfoll} \rlam(t,\statefoll)
	\right] \right).
	\end{align}

		\emph{c)} The formal equation for $\rlam$ is obtained upon integration of~\eqref{Eq:Summary4}:
	\begin{align*}\begin{split}
	0=&\partial_t  \left(\hMOO\rlam\right) \\&+ \int\lagrOMfoll \div_\statefoll  \left( \hMOO \fOMO_2 \int \left[\dynfoll(\statefoll,\hat{\statefoll})(\hat{\statefoll}-\statefoll)-\frac{1}{\regfoll}\lagrOMfoll \right] \hMOOhat\fOMOhat_2\,\d\hat{\statefoll}\,\d\hat{\lagrOMfoll}\right)\,\d\lagrOMfoll \\
	&+ \int\lagrOMfoll \div_\lagrOMfoll  \left( \hMOO \fOMO_2 \int 
	G_2(\statelead,\hat{\statelead}, \crtllead)
	\hMOOhat\fOMOhat_2\,\d\hat{\statefoll}\,\d\hat{\lagrOMfoll}\right)\,\d\lagrOMfoll.
	\end{split}
	\end{align*}
	Integration by parts yields:
	\begin{align*}\begin{split}
		0=\partial_t  \left(\hMOO\rlam\right) &+ \div_\statefoll\left(
		\hMOO \left[
			\int \dynfoll(\statefoll,\hat{\statefoll})(\hat{\statefoll}-\statefoll) \hMOOhat\fOMOhat_2\fOMO_2 \,\d\hat{\statefoll}\,\d\hat{\lagrOMfoll}\,\d\lagrOMfoll\right.\right.\\
			&~~~~~~~~~~~~~~~~\left.\left.- \frac{1}{\regfoll}\int \lagrOMfoll^2  \hMOOhat\fOMOhat_2\fOMO_2 \,\d\hat{\statefoll}\,\d\hat{\lagrOMfoll}\,\d\lagrOMfoll
		  \right]
		\right)\\
		&- \hMOO \int   \fOMO_2 
		G_2(\statelead,\hat{\statelead}, \crtllead)
		\hMOOhat\fOMOhat_2\,\d\hat{\statelead}\,\d\lagrOMfoll,\\
	=\partial_t  \left(\hMOO\rlam\right) &+ \div_\statefoll\left(
	\hMOO \left[ \rlam\int 
	\dynfoll(\statefoll,\hat{\statefoll})(\hat{\statefoll}-\statefoll)^\top
	 \hMOOhat  \,\d\hat{\statefoll}- \frac{1}{\regfoll}\int \lagrOMfoll^2 \fOMO_2 \,\d\lagrOMfoll
	\right]
	\right)\\
	&- \hMOO \int   \fOMO_2  
	G_2(\statelead,\hat{\statelead}, \crtllead)
	\hMOOhat\fOMOhat_2\,\d\hat{\statelead}\,\d\lagrOMfoll.
	\end{split}
	\end{align*}
Using the assumption in~\eqref{Eq:VarZeroCond}, i.e.:
\begin{align*}
	\int \lagrOMfoll^2 \fOMO_2 \,\d\lagrOMfoll=\left(\int \lagrOMfoll \fOMO_2 \,\d\lagrOMfoll\right)^2,
\end{align*}
 we obtain the following equality:
\begin{align*}\begin{split}
0=&\partial_t  \left(\hMOO\rlam\right) + \div_\statefoll\left(
\hMOO \rlam \left[ \int 
\dynfoll(\statefoll,\hat{\statefoll})(\hat{\statefoll}-\statefoll)^\top
\hMOOhat  \,\d\hat{\statefoll}- \frac{1}{\regfoll}\rlam
\right]
\right)\\
&~~~~~~~~~~~~~~~~ - \hMOO \int   \fOMO_2 \int 
G_2(\statelead,\hat{\statelead}, \crtllead)
\hMOOhat\fOMOhat_2\,\d\hat{\statelead}\,\d\lagrOMfoll,\\
	=&\rlam\left[ \partial_t \hMOO+\div_\statefoll\left( \hMOO\left[ \int 
	\dynfoll(\statefoll,\hat{\statefoll})(\hat{\statefoll}-\statefoll)
	\hMOOhat  \,\d\hat{\statefoll}- \frac{1}{\regfoll}\rlam \right]\right) \right]\\
	&+\hMOO \left[ \partial_t \rlam +\nabla_\statefoll\rlam \cdot\left( \int 
	\dynfoll(\statefoll,\hat{\statefoll})(\hat{\statefoll}-\statefoll)
	\hMOOhat  \,\d\hat{\statefoll}- \frac{1}{\regfoll}\rlam
	\right)\right.\\
	& ~~~~~~~~~~~~~~  \left. -\int   \fOMO_2 
	G_2(\statelead,\hat{\statelead}, \crtllead)
	\hMOOhat\fOMOhat_2\,\d\hat{\statelead}\,\d\lagrOMfoll\right].
\end{split}
\end{align*}
Due to~\eqref{Eq:Step3b}, we get on the support of $\hMOO$ using the definition of $G_2$:
\begin{align*}
\begin{split}
	0=&\partial_t \rlam +\nabla_\statefoll\rlam \left( \int 
	\dynfoll(\statefoll,\hat{\statefoll})(\hat{\statefoll}-\statefoll)^\top
	\hMOOhat  \,\d\hat{\statefoll}- \frac{1}{\regfoll}\rlam
	\right) \\
&	-\int   \fOMO_2
\left[
-\D_\statefoll\tilde{m}(\statefoll)^\top \nabla_m\objfoll(\tilde{m}(\statefoll);v) -\D_{\statefoll}\left[ \dynfoll(\statefoll,\hat{\statefoll})(\hat{\statefoll}-\statefoll)\right]^\top \lagrOMfoll \right.
\\
&\left.~~~~~~~~~~~~~~~~-\D_{{\statefoll}}\left[ \dynfoll(\hat{\statefoll},\statefoll)(\statefoll-\hat{\statefoll})\right]^\top \hat{\lagrOMfoll}
\right]
	\hMOOhat\fOMOhat_2\,\d\hat{\statelead}\,\d\lagrOMfoll.
	\end{split}
\end{align*}
Using that $\int\fOMOhat_1\,\d\hat{ \statefoll}=\int\fOMOhat_2\,\d\hat{ \psi}=1$ yields:
\begin{align}\label{Eq:Step3c}
\begin{split}
0=&\partial_t \rlam +\nabla_\statefoll\rlam \left( \int 
\dynfoll(\statefoll,\hat{\statefoll})(\hat{\statefoll}-\statefoll)^\top
\hMOOhat  \,\d\hat{\statefoll}- \frac{1}{\regfoll}\rlam
\right) 
+\D_\statefoll\tilde{m}(\statefoll)^\top \nabla_m\objfoll(\tilde{m}(\statefoll);v)\\
&+\int \D_{\statefoll}\left[ \dynfoll(\statefoll,\hat{\statefoll})(\hat{\statefoll}-\statefoll)\right]^\top \rlam \hMOOhat\,\d\hat{\statefoll} +\int\D_{\hat{\statefoll}}\left[ \dynfoll(\hat{\statefoll},\statefoll)(\statefoll-\hat{\statefoll})\right]^\top\hat{\rlam} \hMOOhat\,\d\hat{\statefoll}.
\end{split}
\end{align}
	We leave this  and come back to Equation~\eqref{Eq:Step3c} in Step 6 where we connect it with the optimality system of~\eqref{Eq:Summary3}.
	
\emph{Step 4.}
Due to Lemma~\ref{Lem:2}, the corresponding mean-field formulation of the optimal control problem of the followers in~\eqref{Eq:Stack} is obtained and the moment $m$ of the followers in the leader's objective functional is rewritten by the density $\fMOO$.
This yields the mean-field Stackelberg game:
\begin{align}\label{Eq:Summary1}
\begin{split}
\min\limits_{\crtllead} & ~\int\limits_\tini^\tend \left[\objlead(\crtllead,m_{\fMOO}(t) +\frac{\reglead}{2} \crtllead^2\right]\,\d t\\
\st 
&~\min\limits_{\crtlfoll,\fMOO}~ \int\limits_{\tini}^\tend\int \left[\objfoll(\tilde{m}(\statefoll);\crtllead)+\frac{\regfoll}{2} \crtlfoll^2\right]\fMOO\,\d \statefoll\,\d t\\
& ~~~~ \st ~~0=\partial_t\fMOO+\div_\statefoll \left(\fMOO \int \left[\dynfoll(\statefoll,\hat{\statefoll})(\hat{\statefoll}-\statefoll)+\crtlfoll\right] \fMOOhat \,\d\hat{\statefoll}   \right)\\
&~~~~~~~~~~\fMOO(\tini,\statefoll)=\fMOO_\tini(\statefoll).
\end{split}
\end{align}
We can identify
	\begin{align*}
	\fOMO(t,\statefoll,\lagrOMfoll)&=\fMOO(t,\statefoll)\fOMO_2(t,\statefoll,\lagrOMfoll),\\
	-\int\lagrOMfoll\fOMO_2(t,\statefoll,\lagrOMfoll)\,\d\lagrOMfoll&=\nabla_\statefoll \lagrMOfoll(t,\statefoll),
	\end{align*}
and we obtain equivalence of~\eqref{Eq:Stack} and~\eqref{Eq:Summary1} according to Lemma~\ref{Lem:2} in the mean-field limit.

\emph{Step 5.}
	The first-order optimality conditions for the follower problem in~\eqref{Eq:Summary1} are given by~\eqref{Eq:Summary3} which yields the mean-field leader's problem as:
	\begin{align}\label{Eq:Summary3}
	\begin{split}
	\min\limits_{\crtllead} & ~\int\limits_\tini^\tend \left[\objlead(\crtllead,m_{\fMOO}(t)) +\frac{\reglead}{2} \crtllead^2\right]\,\d t\\
	\st &~0=\partial_t\fMOO+\div_\statefoll \left(\fMOO \int \left[\dynfoll(\statefoll,\hat{\statefoll})(\hat{\statefoll}-\statefoll)+\frac{1}{\regfoll}\nabla_\statefoll\Psi\right] \fMOOhat \,\d\hat{\statefoll}   \right)\\
	&~0=\partial_t \Psi(t,\statefoll) 
	+\int\left\{ \left[\dynfoll(\statefoll,\hat{\statefoll})(\hat{\statefoll}-\statefoll)+\frac{1}{\regfoll}\nabla_\statefoll\Psi\right]^\top \nabla_\statefoll \Psi 
	\right.\\&~~~~~~ +\left. \left[\dynfoll(\hat{\statefoll},\statefoll)(\statefoll-\hat{\statefoll})+\frac{1}{\regfoll}\nabla_{\hat{\statefoll}}\hat{\Psi}\right]^\top \nabla_{\hat{\statefoll}} \hat{\Psi}\right\} \fMOOhat\,\d\hat{\statefoll}
	\\
	&~~~~~~ -\objfoll(\tilde{m}(\statefoll),\crtllead)-\frac{1}{2\regfoll}\left(\nabla_\statefoll\Psi\right)^2\\
	&~\fMOO(\tini,\statefoll)=\fMOO_\tini(\statefoll), \quad\Psi(\tend,\statefoll)=0,
	\end{split}
	\end{align}	
where the follower control is given by:
\begin{align*}
\crtlfoll(t,\statefoll)=\frac{1}{\regfoll}\nabla_\statefoll\Psi(t,\statefoll).
\end{align*}

\emph{Step 6.} 
 Finally, we show the equivalence of~\eqref{Eq:Summary4} and~\eqref{Eq:Summary3}.
 
\emph{a)}
For $\hMOO(t,\statefoll)=\fMOO(t,\statefoll)$, we have that the objective functionals in~\eqref{Eq:Step3a} and  in~\eqref{Eq:Summary3} coincide.

\emph{b)} 
	The first constraint in~\eqref{Eq:Summary3} and~\eqref{Eq:Step3b} are equivalent provided that:
	\begin{align}\label{Eq:Consitency2}
		\rlam(t,\statefoll)=-\nabla_\statefoll \lagrMOfoll(t,\statefoll).
	\end{align}

\emph{c)} 
		Due to the constraint in~\eqref{Eq:Summary3}, the gradient $\nabla_\statefoll\Psi$ formally fulfills:
\begin{align*}
\begin{split}
0=&\partial_t \nabla_\statefoll \Psi(t,\statefoll) 
+\nabla_\statefoll\int\left\{ \left[\dynfoll(\statefoll,\hat{\statefoll})(\hat{\statefoll}-\statefoll)+\frac{1}{\regfoll}\nabla_\statefoll\Psi\right]^\top \nabla_\statefoll \Psi
\right.\\ &\left.
+ \left[\dynfoll(\hat{\statefoll},\statefoll)(\statefoll-\hat{\statefoll})+\frac{1}{\regfoll}\nabla_{\hat{\statefoll}}\hat{\Psi}\right]^\top \nabla_{\hat{\statefoll}} \hat{\Psi}\right\} \fMOOhat\,\d\hat{\statefoll}
\\
&
 -\D_\statefoll\tilde{m}(\statefoll)^\top\objfoll(\tilde{m}(\statefoll),\crtllead)-\frac{1}{\regfoll}\nabla^2_\statefoll\Psi\nabla_\statefoll\Psi,\\
=&\partial_t \nabla_\statefoll \Psi(t,\statefoll)  + \int \left\{\D_\statefoll \left[\dynfoll(\statefoll,\hat{\statefoll})(\hat{\statefoll}-\statefoll)\right]^\top\nabla_\statefoll \Psi+\dynfoll(\statefoll,\hat{\statefoll})(\hat{\statefoll}-\statefoll)\nabla^2_\statefoll\Psi+\frac{2}{\regfoll}\nabla^2_\statefoll\Psi\nabla_\statefoll\Psi\right.
\\&\left.+\D_\statefoll\left[\dynfoll(\hat{\statefoll},\statefoll)(\statefoll-\hat{\statefoll})\right]^\top \nabla_{\hat{\statefoll}}\hat{\Psi}\right\}\fMOOhat\,\d\hat{ \statefoll} -\D_\statefoll\tilde{m}(\statefoll)^\top\objfoll(\tilde{m}(\statefoll),\crtllead)-\frac{1}{\regfoll}\nabla^2_\statefoll\Psi\nabla_\statefoll\Psi,
\end{split}
\end{align*}
and therefore:
\begin{align}\label{Eq:Step3d}
\begin{split}
0=&\partial_t \nabla_\statefoll \Psi(t,\statefoll) \\
&+ \int \left\{\D_\statefoll \left[\dynfoll(\statefoll,\hat{\statefoll})(\hat{\statefoll}-\statefoll)\right]^\top\nabla_\statefoll \Psi+\dynfoll(\statefoll,\hat{\statefoll})(\hat{\statefoll}-\statefoll)\nabla^2_\statefoll\Psi
\right.\\&\left.
~~~~~~~~~+\D_\statefoll\left[\dynfoll(\hat{\statefoll},\statefoll)(\statefoll-\hat{\statefoll})\right]^\top \nabla_{\hat{\statefoll}}\hat{\Psi}\right\}\fMOOhat\,\d\hat{ \statefoll}
\\
&
 -\D_\statefoll\tilde{m}(\statefoll)^\top\objfoll(\tilde{m}(\statefoll),\crtllead)+\frac{1}{\regfoll}\nabla^2_\statefoll\Psi\nabla_\statefoll\Psi.
\end{split}
\end{align}
	Equations~\eqref{Eq:Step3c} and~\eqref{Eq:Step3d} coincide provided that the consistency condition in~\eqref{Eq:Consitency2} holds true.

Hence, the optimality conditions coincide under the condition of the Theorem.
This finishes the proof.


\section{Solution Method}\label{Sec:4}

We propose an iterative scheme, that alternates between updating the leader control $\crtllead$ and solving a set of PDEs.

In the previous section we have proven that under the given conditions the order of the optimization steps and the mean-field limit can be  exchanged.
However, the resulting optimization systems differ considerably.

In Option 1, the OOM approach, the optimality system consists of a PDE and an algebraic condition, the unknown is the probability density $\fOOM(t,\zeta,\theta)$ which is a $4n_F+1$-dimensional function.
The optimality system for Option 2 (OMO), consists of two coupled PDEs in the probability density $\fOMO(t,\zeta)$ and the costate $\Theta(t,\zeta)$.
Both depend on $2n_F+1$ variables.

We therefore develop an algorithm for Option 3, the MOO approach,  which consists of four coupled PDEs.
The unknowns  are $\fMOO(t,\statefoll)$, $\Psi(t,\statefoll)$, $\Phi_1(t,\statefoll)$, and $\Phi_2(t,\statefoll)$ and they depend on only $n_F+1$ variables.
The optimality system derived by the MOO approach reads:
\begin{subequations}\label{Eq:FullOptMOO}
	\begin{align}
	0=&\partial_t\fMOO+\div_\statefoll \left(\fMOO \int \left[\dynfoll(\statefoll,\hat{\statefoll})(\hat{\statefoll}-\statefoll)+\frac{1}{\regfoll}\nabla_\statefoll\Psi\right] \fMOOhat \,\d\hat{\statefoll}   \right),\label{Eq:FullOptMOO1}\\
	\begin{split}\label{Eq:FullOptMOO2}
	0=&\partial_t \Psi +	\int \left[\dynfoll(\statefoll,\hat{\statefoll})(\hat{\statefoll}-\statefoll)^\top \nabla_\statefoll \Psi+ \dynfoll(\hat{\statefoll},\statefoll)(\statefoll-\hat{\statefoll})^\top \nabla_{\hat{\statefoll}} \hat{\Psi}\right] \fMOOhat\,\d\hat{\statefoll}
	\\
	&
	 +\frac{1}{2\regfoll}\left(\nabla_\statefoll\Psi\right)^2- \objfoll(\tilde{m}(\statefoll),\crtllead),
	\end{split}\\
	\begin{split}	\label{Eq:FullOptMOO3}
	0=&\partial_t \Phi_1
	+ \frac{1}{\regfoll}\nabla_\statefoll \Phi_1^\top \nabla_\statefoll\Psi\\ &+	\int \left[\dynfoll(\statefoll,\hat{\statefoll})(\hat{\statefoll}-\statefoll)^\top \nabla_\statefoll \Phi_1+ \dynfoll(\hat{\statefoll},\statefoll)(\statefoll-\hat{\statefoll})^\top \nabla_{\hat{\statefoll}} \hat{\Phi}_1\right] \fMOOhat\,\d\hat{\statefoll}\\
	&+	\int \left[\dynfoll(\statefoll,\hat{\statefoll})(\hat{\statefoll}-\statefoll)^\top \nabla_\statefoll \Psi+ \dynfoll(\hat{\statefoll},\statefoll)(\statefoll-\hat{\statefoll})^\top \nabla_{\hat{\statefoll}} \hat{\Psi}\right] \hat{\Phi}_2\,\d\hat{\statefoll}
	\\&- \nabla_m \objlead(\crtllead,m_{\fMOO} (t))\tilde{m}(\statefoll), \end{split}\\
	\begin{split}\label{Eq:FullOptMOO4}
	0=&\partial_t \Phi_2+\div_\statefoll \left( \int \dynfoll(\statefoll,\hat{\statefoll})(\hat{\statefoll}-\statefoll)\fMOOhat \,\d\hat{ \statefoll} \Phi_2\right)
	\\&+\div_\statefoll \left( \int \dynfoll(\statefoll,\hat{\statefoll})(\hat{\statefoll}-\statefoll) \hat{\Phi}_2\,\d\hat{ \statefoll}\fMOO \right)
	\\
	&
	 + \frac{1}{\regfoll} \div_\statefoll \left( \nabla_\statefoll\Psi \Phi_2\right)  - \frac{1}{\regfoll} \div_\statefoll \left( \nabla_\statefoll\Phi_1 \fMOO\right), 
	\end{split}\\
	0=&\nabla_\crtllead\objlead(\crtllead,m_{\fMOO}(t) ) +\reglead \crtllead -\int \nabla_\crtllead \objfoll (\tilde{m}(\statefoll),\crtllead)\Phi_2  \,\d\statefoll,\label{Eq:FullOptMOO5}\\
	& \fMOO(\tini,\statefoll)=\fMOO_\tini(\statefoll), \quad\Psi(\tend,\statefoll)=0, \quad \Phi_1(\tend,\statefoll)=0, \quad\Phi_2(\tini,\statefoll)=0.
	\end{align}
\end{subequations}
Note, Equations~\eqref{Eq:FullOptMOO1} and~\eqref{Eq:FullOptMOO4} are forward in time, while~\eqref{Eq:FullOptMOO2} and~\eqref{Eq:FullOptMOO3} are backwards in time.

We solve the equation~\eqref{Eq:FullOptMOO2} for $\nabla_\statefoll\Psi$ instead of $\Psi$, likewise we solve~\eqref{Eq:FullOptMOO3} for $\nabla_\statefoll\Phi_1$ instead of $\Phi_1$.
Then all PDEs are nonlinear conservative transport equations.
These PDEs are integro-differential equations and  have stiff source terms as  $\regfoll\rightarrow0$.

Observe,  that the equation for the gradient $\nabla_\statefoll \Psi$ is independent of the other unknowns in case $P(\statefoll,\hat{\statefoll})= 0$.
Then we have:
\begin{subequations}\label{Eq:MOOP01}
	\begin{align}
	0=&\partial_t \nabla_\statefoll \Psi(t,\statefoll) + \nabla_\statefoll\left( \frac{1}{2\regfoll} (\nabla_\statefoll\Psi(t,\statefoll))^2 -\objfoll(\tilde{m}(\statefoll),\crtllead(t))  \right),\label{Eq:MOOP01a}\\
	0=&\nabla_\statefoll\Psi(T,\statefoll).\label{Eq:MOOP01b}
	\end{align}
\end{subequations}
This equation also determines the followers control which we specify in the lemma below.
\begin{lemma}
	Consider the Stackelberg game given by~\eqref{Eq:Stack}. Then the follower control is:
	 \begin{equation*}
		\crtlfoll(t,\statefoll)=\frac{1}{\regfoll}\nabla_\statefoll\Psi(t,\statefoll).
	\end{equation*}
	If the interaction kernel is $P(\statefoll,\hat{ \statefoll})=0$,
	then the evolution of the optimal follower's control $\crtlfoll$ is uniquely determined by Equation~\eqref{Eq:MOOP01} for any given leader control $\crtllead$.
\end{lemma}
\begin{proof}
	Lagrange multiplier theorem is applied to lower level problem of~\eqref{Eq:Summary1}, see Step 5 of the proof of Theorem~\ref{Th:MainRes}.
\end{proof}

For $P\equiv0$, we propose the following sequential solution to~\eqref{Eq:FullOptMOO}:
The solution to Equation~\eqref{Eq:MOOP01} allows to solve the conservation law for $\fMOO$:
\begin{subequations}\label{Eq:MOOP02}
\begin{align}
0=&\partial_t\fMOO(t,\statefoll)+\div_\statefoll \left(\frac{1}{\regfoll}  \nabla_\statefoll\Psi(t,\statefoll)  \fMOO(t,\statefoll) \right), \label{Eq:MOOP02a}\\
\fMOO(\tini,\statefoll)=&\fMOO_\tini(\statefoll).\label{Eq:MOOP02b}
\end{align}
\end{subequations}
Also, we solve a balance  equation for $\nabla_\statefoll \Phi_1$:
\begin{subequations}\label{Eq:MOOP03}
	\begin{align}\begin{split}
		0=&\partial_t\nabla_\statefoll \Phi_1(t,\statefoll)
	+ \nabla_\statefoll\left(\frac{1}{\regfoll}\nabla_\statefoll \Phi_1(t,\statefoll)^\top \nabla_\statefoll\Psi(t,\statefoll)\right)
	\\&- \nabla_m \objlead(\crtllead(t),m_{\fMOO} (t))\nabla_\statefoll\tilde{m}(\statefoll),\label{Eq:MOOP03a}
	\end{split}\\
	0=&\nabla_\statefoll\Phi_1(T,0).\label{Eq:MOOP03b}
	\end{align}
\end{subequations}
Finally, we solve for $\Phi_2$ by \eqref{Eq:FullOptMOO4}:
\begin{subequations}\label{Eq:MOOP04}
	\begin{align}
	0=&\partial_t \Phi_2(t,\statefoll) + \frac{1}{\regfoll} \div_\statefoll \left( \nabla_\statefoll\Psi(t,\statefoll) \Phi_2(t,\statefoll)\right)  - \frac{1}{\regfoll} \div_\statefoll \left( \nabla_\statefoll\Phi_1(t,\statefoll) \fMOO(t,\statefoll)\right) ,\label{Eq:MOOP04a}\\
	0=&\Phi_2(0,\statefoll).\label{Eq:MOOP04b}
	\end{align}
\end{subequations}
This procedure sequentially solves the optimality system for any fixed leader control $\crtllead$.

 Then, Equation~\eqref{Eq:FullOptMOO5} implicitly describes the optimal leader control depending on the objective function of the followers and the leader as well as the dual variable $\Phi_2$.
As last step we update of the leader control: $$\crtllead^{k+1}(t)=\crtllead^k(t)+\sigma_kd^k(t)$$
where $d(t)\in\R^{n_L}$ for $t\in[0,T]$  is computed by:
\begin{equation*}
d(t)=-\left( \nabla_\crtllead\objlead(\crtllead(t),m_{\fMOO}(t) ) +\reglead \crtllead(t) -\int \nabla_\crtllead \objfoll (\tilde{m}(\statefoll),\crtllead(t))\Phi_2(t,\statefoll)  \,\d\statefoll\right).
\end{equation*} 
We combine this with a backtracking line search for the step size $\sigma_k>0$  based on the Armijo condition for the leader objective, c.f. \cite[Cor.2]{Armijo1966}.
The complete procedure is summarized  in Algorithm ~\ref{Alg:1}.

\begin{algorithm}[H]
	\caption{Continuous Optimization Method }\label{Alg:1}
	\begin{algorithmic}[1]
		\State \textbf{Initialize} Choose initial guess for the leader control $\crtllead^0(t)\in\R^{n_L}$ for $t\in[0,T]$.
		\For{$k=0,1,\dots$}
		\State Solve backward Equation~\eqref{Eq:MOOP01} with $\crtllead^{k}$ to get $\nabla_\statefoll \Psi^{k}$.
		\State Solve forward Equation~\eqref{Eq:MOOP02} to get $\fMOO_{k}$.
		\State Solve backward Equation~\eqref{Eq:MOOP03}  to get $\nabla_\statefoll \Phi_1^{k}$.
		\State Solve forward Equation~\eqref{Eq:MOOP04}  to get $ \Phi_2^{k}$.		
		\State Compute a decent direction 
		\begin{equation*}
			d^k(t)=-\left( \nabla_\crtllead\objlead(\crtllead^k(t),m_{\fMOO_k}(t) ) +\reglead \crtllead^k(t) -\int \nabla_\crtllead \objfoll (\tilde{m}(\statefoll),\crtllead^k(t))\Phi^k_2(t,\statefoll)  \,\d\statefoll\right).
		\end{equation*} 
		\State Choose a step size $\sigma_k>0$ according to Armijo rules.
		\State Update control $\crtllead^{k+1}(t)=\crtllead^k(t)+\sigma_kd^k(t)$.
		\If {The termination condition is satisfied,} STOP.\EndIf
		\EndFor 
	\end{algorithmic}
\end{algorithm}

For all PDEs we use periodic boundary conditions.
The state space is discretize by equidistant cells, we use $n_\statefoll$ to denote the number of cells in a state space dimension of $\statefoll$.
We use  Lax-Friedrich's method to describe  the  numerical flux across the cell interfaces~\cite{LeVeque1992}.
The integral in $\objlead$ is discretized by a first  order  quadrature rule.

The stopping criteria of the algorithm is that 
the change in the leader control is below a relative tolerance.


\section{Numerical Results} \label{Sec:5}

\subsection{Setup}
In this section we discuss the results of the numerical experiments  that were implemented in Matlab.
We start with the description of the used  parameters:
The dimension of the state and control space of the followers  is  $n_F=1$ and we assume  $\statefoll\in[0,2]$.
Also the leader control is one-dimensional and we consider the time time $T=1$.

The leader objective combines tracking of the leader control with  moment of the followers as:
 $$\objlead(\crtllead,m_{\fMOO}(t))=\frac{1}{2}(\crtllead_{\mathrm{d}}+m_{\fMOO}(t)-\crtllead)^2,$$ with the desired control: $$\crtllead_{\mathrm{d}}(t)=\sin(2\pi t).$$
 The followers objective couple the states of the followers and the leader control as:
 \begin{equation}\label{Eq:FollObj}
 \objfoll(\tilde{m}(\statefoll),\crtllead)=-\frac{1}{2}(\tilde{m}(\statefoll)-\crtllead)^2.
 \end{equation}
 Both objective functions are regularized by a quadratic term, the regularization parameters are
 $\reglead,\regfoll$ in the experiments.
 We consider the expectation $\tilde{m}(\statefoll)=\statefoll$. 
 As initial condition for the probability density $\fMOO$ we chose the uniform distribution: 
 $$\fMOO_0(\statefoll) =\chi_{[0.5,1.5]}(\statefoll).$$
The state space is discretized by $n_\statefoll$ equidistant cells.
The time steps of the PDE solving are chosen to satisfy the Courant-Friedrichs-Lewy condition with the constant $\mathrm{CFL}=0.95$.
For the evaluation of the objective functions and the termination conditions the time is discretized by $n_t$  equidistant points in time. 
Whenever necessary linear reconstruction is performed by  interpolation.
Further, we choose a maximum iteration number for Algorithm~\ref{Alg:1}.
$\mathrm{maxIter}=100$.
The tolerance is related to the space discretization as
$\mathrm{tol}=\frac{1}{100n_\statefoll}$.
As initial guess: $$\crtllead^0(t)=t,$$ is chosen.

\subsection{Experiments}
In Figure~\ref{Fig:CPU} on the left, the behavior of the solutions is illustrated  for different discretizations of the leader control $\crtllead$.
We observe  monotonically decrease of the $L^2$ error as the control grid increases.
On the right,  the implementation of Algorithm~\ref{Alg:1} is compared  to fminuncon, a commercial trust-region solver in Matlab where we provide gradient information based on same PDE solvers.
The implementation outperforms fminuncon in terms of computational time for all mesh sizes and the computational time increases linearly with mesh refinement in Algorithm~\ref{Alg:1}.

\begin{figure}[]
	\begin{center}
		\begin{minipage}{0.5\textwidth}
			\includegraphics[width=\textwidth]{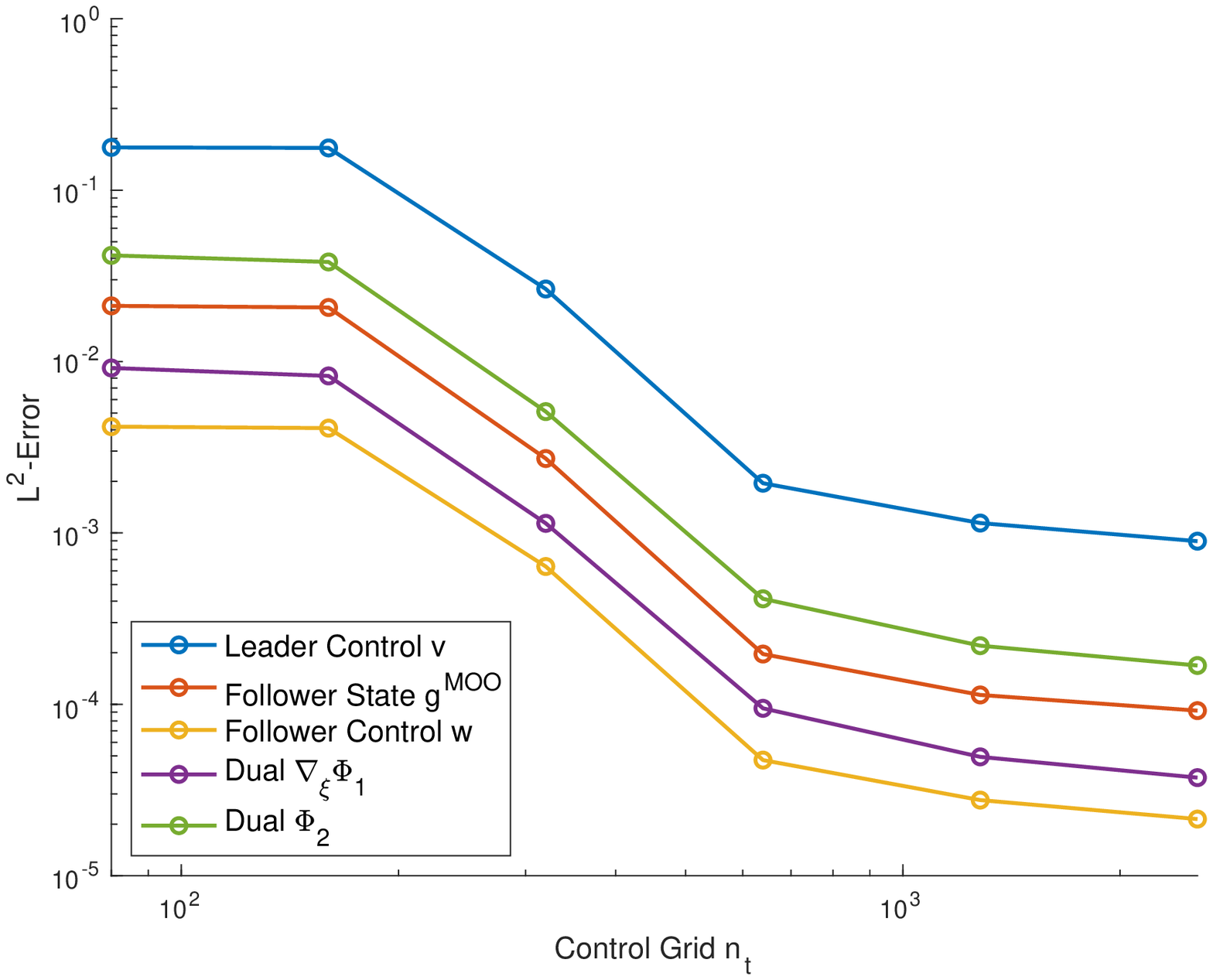}
		\end{minipage}\begin{minipage}{0.5\textwidth}
			\includegraphics[width=\textwidth]{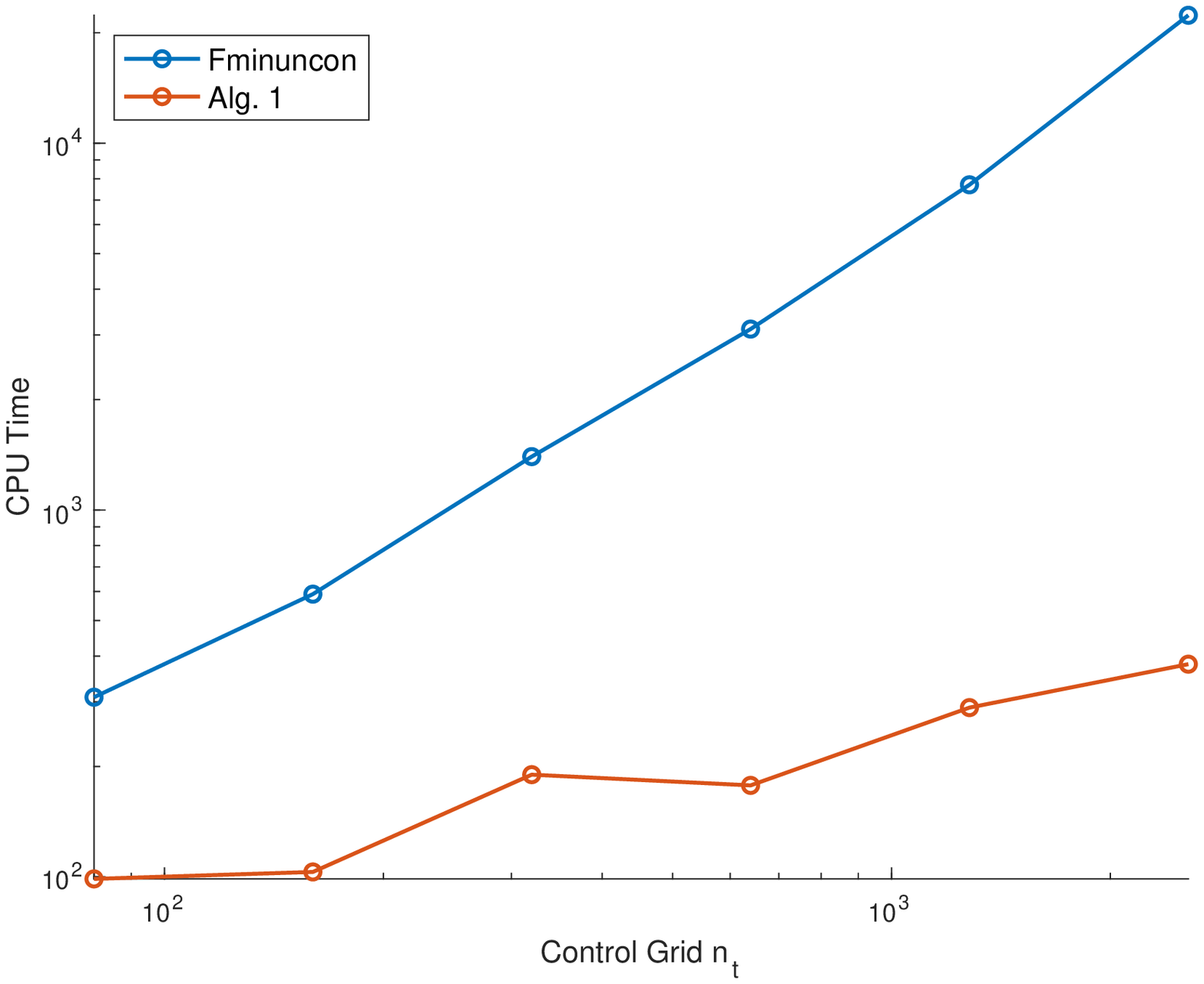}
	\end{minipage}

	\end{center}

	\caption{Performance. On the left, the	L$^2$-error of the leader control and the PDE solutions is plotted for refined mesh  $n_t$ of the leader control.
		The solutions are compared to the solution of the finest mesh.
		On the right, CPU times of Matlab fiminuncon trust-region with user specified gradient and an implementation of Algorithm~\ref{Alg:1} are compared.
	 }\label{Fig:CPU}
\end{figure}

%
%

All experiments in the plots are performed  for the space discretization of $n_\statefoll=500$.

In Figure~\ref{Fig:LeaderCrtl}, the leader control is illustrated.
The control action is larger for smaller regularization parameters due to decreasing  influence of the objective.
For any regularization, the sinodal shape of the optimal leader control $v$ is clearly recognizable.
This corresponds to the expected control $v_\mathrm{d}$.

\begin{figure}[]
	\begin{center}
	\includegraphics[width=0.48\textwidth]{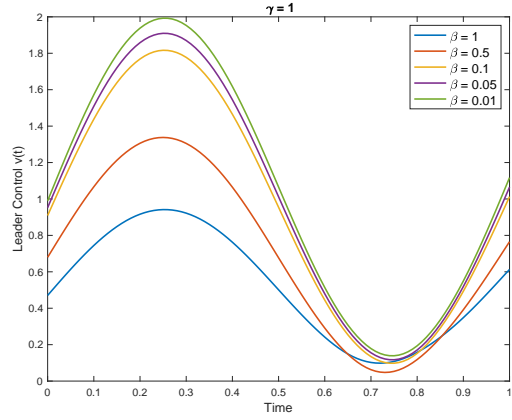}
		\end{center}
	\caption{The optimal leader control is shown depending on the choice of the regularization parameter.
		\label{Fig:LeaderCrtl}}
\end{figure}

In Figure~\ref{Fig:Foll} on the left, the control of the followers is shown as surface plot.
The final time condition $\nabla_\statefoll\Psi(T,\statefoll)=0$ is visible in the control plot.
On the right the resulting evolution of the followers' state is plotted.
Note, the choice of the followers' objective in~\eqref{Eq:FollObj}, give them the initiative to move away from the leader control.
Due to this, we observe that the followers tend to concentrate at $\statefoll=2$ and $\statefoll=0$ as time evolves.

\begin{figure}[]
	\begin{center}
		\includegraphics[width=0.7\textwidth]{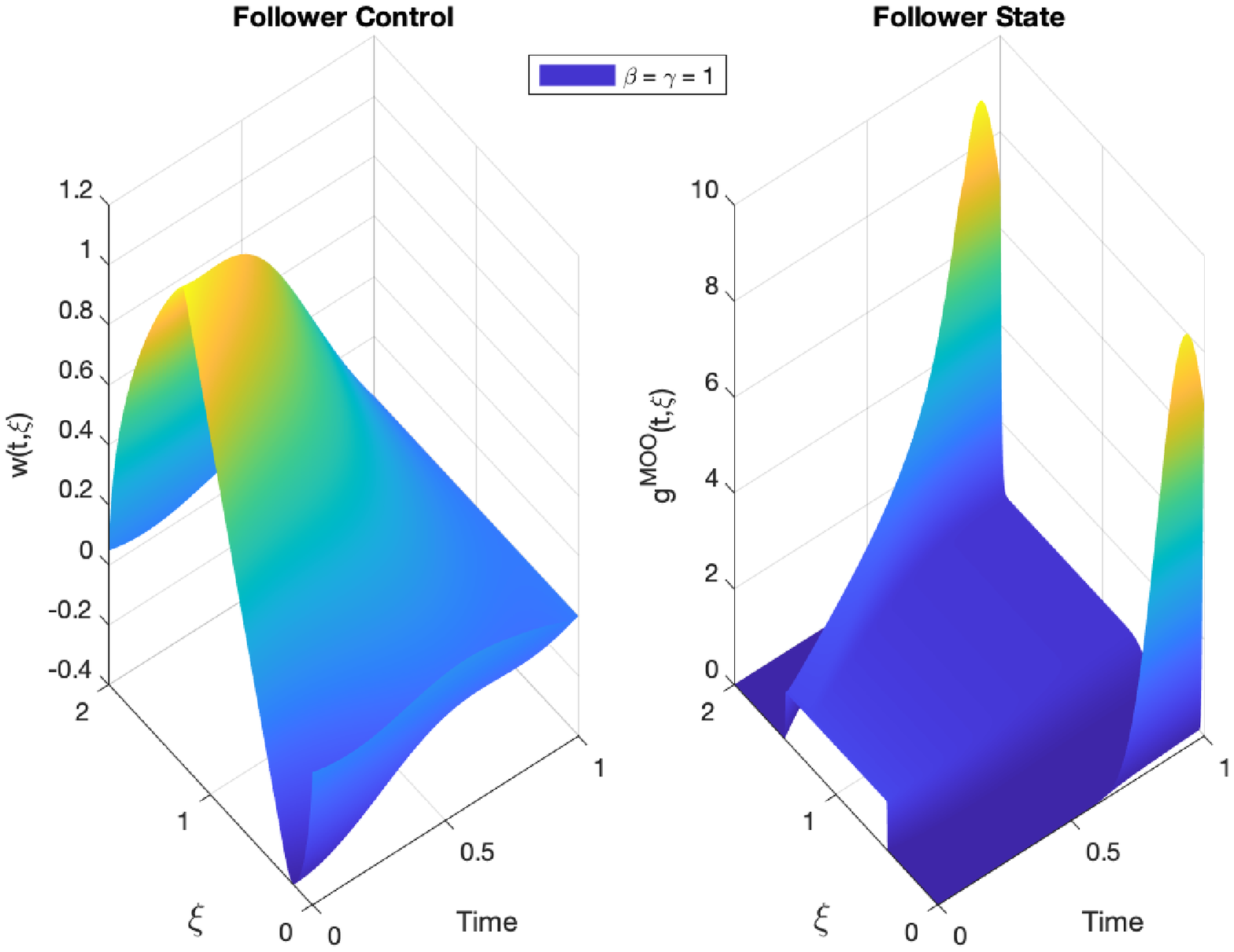}
	\end{center}
	
	\caption{Time evolution of the followers' control and state.
		On the left, the solution of Equation~\eqref{Eq:MOOP01}, which is   $\crtlfoll(t,\statefoll)=\frac{1}{\regfoll}\nabla_{\statefoll}\Psi$, is shown.
		On the right, it is the solution of~\eqref{Eq:MOOP02}.\label{Fig:Foll}
	}
\end{figure}
In Figure~\ref{Fig:Duals}, the solutions of the adjoint equations, that are related to the optimization of the leader, are plotted.
The strong relation to the followers' initial condition is visible on the left.
On the right, we observe a similar shape as the evolution of the followers' state.
Those states are due to the followers' state influence in the source terms of the adjoint equations.
\begin{figure}[]
	\begin{center}
		\includegraphics[width=0.7\textwidth]{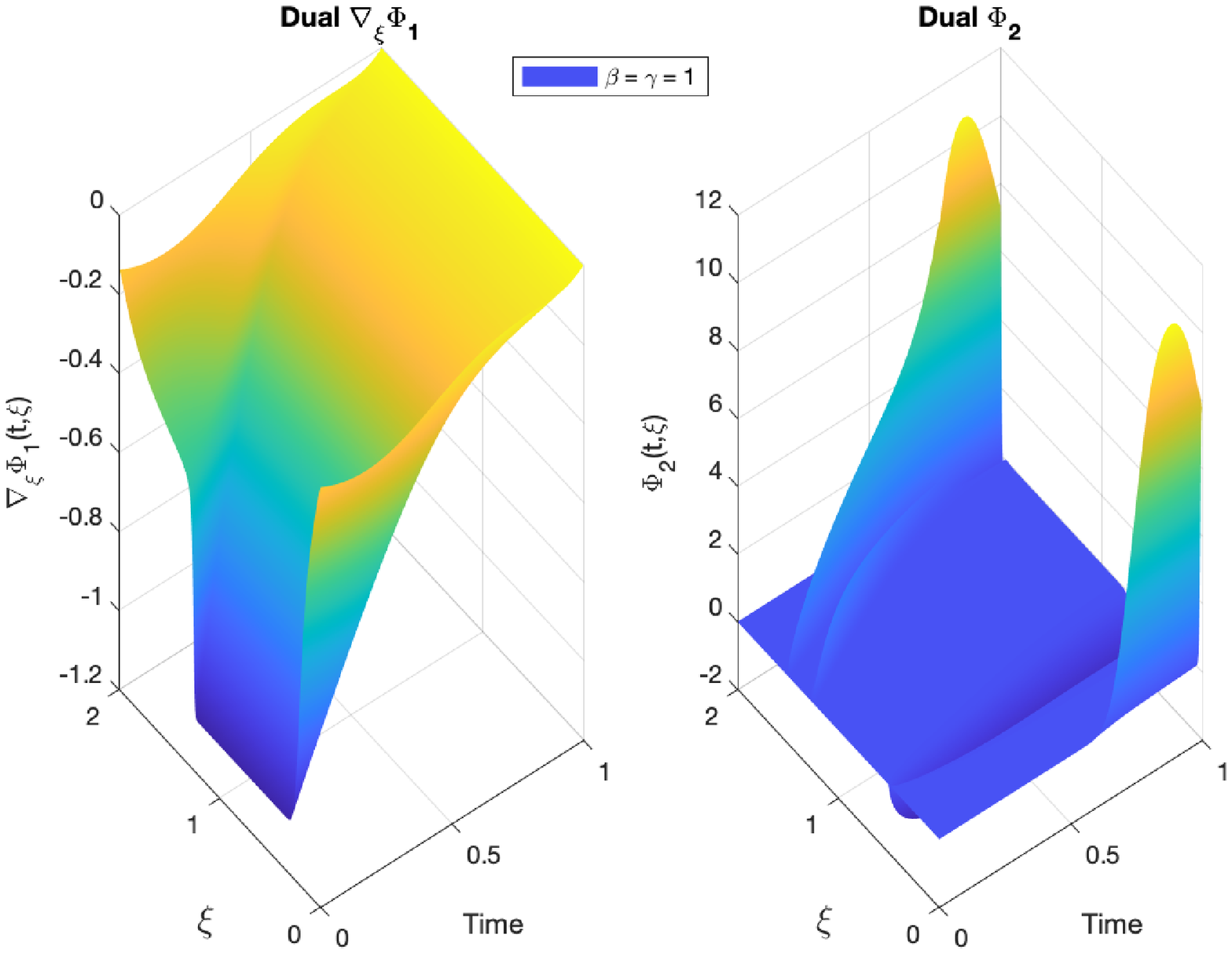}
	\end{center}
	
	\caption{Time evolution of the dual states, on the left is the solution of Equation~\eqref{Eq:MOOP03}, on the right of~\eqref{Eq:MOOP04}. \label{Fig:Duals}}
\end{figure}

\section{Outlook}

In this article a Stackelberg game with  possibly infinitely many of followers is discussed.
The formal mean-field limit is derived for the Stackelberg game and consistency conditions are established.
Then a numerical method for the full optimality system and numerical results are presented.

Future steps include the extension of the Stackelberg of multiple leaders in order to enable accurate analysis of energy markets and tolling in vehicular traffic models.
Also comparisons to particle simulations of the model are subject to further research.
From the mathematical perspective, it is interesting, if Theorem~\ref{Th:MainRes} can be shown also rigorously.

\appendix
\section{Optimality Systems}

If the OOM approach is followed, we get as optimality system of \eqref{Eq:Summary2}  for $i=1,\dots,N$:
\begin{align}\label{Eq:Summary5}
\begin{split}
\dot{\statelead}_i& = \frac{1}{N}\suml_{j=1}^N						G(\statelead_i,\statelead_j,v)\\ 
\dot{\theta}_i& = - \frac{1}{N}\suml_{j=1}^N \left[\D_1 G(\statelead_i,\statelead_j,v) \theta_i +\D_2 G(\statelead_j,\statelead_i,v) \theta_j\right] +\begin{bmatrix}
\D_{{\statefoll}_i}\tilde{m}(\statefoll_i)\\0
\end{bmatrix}^\top\nabla_m\objlead(\crtllead,m(\vec{\statefoll}))\\
0&=\nabla_\crtllead \objlead(\crtllead,m\left(\vec{\statefoll}\right))+\reglead\crtllead-\frac{1}{N^2}\suml_{i=1}^N\suml_{j=1}^N \D_\crtllead G(\statelead_i,\statelead_j,\crtllead)\theta_i
\end{split}
\end{align}
Computing the mean field behavior yields the following equations with $\fOOM=\fOOM(t,\statelead,\theta)$:
\begin{align}\label{Eq:Summary8}
\begin{split}
0=&\partial_t\fOOM+\div_\statelead
\left(
\fOOM
\int G(\statelead,\hat{\statelead},\crtllead)\fOOMhat\,\d\hat{ \statelead}\,\d\hat{ \theta}
\right)\\
&-\div_\theta
\left(
\fOOM
\int \left[\D_1G(\statelead,\hat{\statelead},\crtllead)\theta+\D_2G(\hat{\statelead},\statelead,\crtllead)\hat{\theta}\right]
\fOOMhat\,\d\hat{ \statelead}\,\d\hat{ \theta}
\right)\\
& +\begin{bmatrix}
\D_{{\statefoll}}\tilde{m}(\statefoll)^\top\nabla_m\objlead(\crtllead,m_{\fOOM}(t))\\0
\end{bmatrix}^\top \nabla_\theta\fOOM\\
0=&\nabla_\crtllead \objlead(\crtllead,m\left(\vec{\statefoll}\right))+\reglead\crtllead-\int \D_\crtllead G(\statelead,\hat{\statelead},\crtllead)\theta\fOOM\fOOMhat\,\d\hat{\statelead}\,\d\hat{\theta}\,\d{\statelead}\,\d{\theta}
\end{split}
\end{align}
Where we specify
\begingroup
\allowdisplaybreaks
\begin{align*}
	\D_1G(\statelead,\hat{\statelead},\crtllead)=&\begin{bmatrix}
	\D_\statefoll\left[\dynfoll(\statefoll,\hat{ \statefoll})(\hat{ \statefoll}-\statefoll)   \right]& \begin{matrix}
	-\frac{1}{\regfoll}& &\\& \ddots &\\ & &-\frac{1}{\regfoll}
	\end{matrix}\\
		\D_\statefoll G_2(\statelead,\hat{\statelead},\crtllead)
	&	-\D_\statefoll\left[\dynfoll(\statefoll,\hat{ \statefoll})(\hat{ \statefoll}-\statefoll)   \right]
	\end{bmatrix}\\
	\D_\statefoll G_2(\statelead,\hat{\statelead},\crtllead) = & -\D^2_{{\statefoll}}\tilde{m}(\statefoll)^\top\nabla_m\objfoll(\tilde{m}(\statefoll),\crtllead) \\&-\left(\D_{{\statefoll}}\tilde{m}(\statefoll)\right)^2\nabla^2_m\objfoll(\tilde{m}(\statefoll),\crtllead)
	-\D^2_\statefoll\left[\dynfoll(\statefoll,\hat{ \statefoll})(\hat{ \statefoll}-\statefoll)   \right] \psi
	\\
		\D_2G(\statelead,\hat{\statelead},\crtllead)=& \begin{bmatrix}
		\D_{\hat{ \statefoll}}\left[\dynfoll(\statefoll,\hat{ \statefoll})(\hat{ \statefoll}-\statefoll)   \right]&\begin{matrix}
		0& &\\& \ddots &\\ & &0
		\end{matrix}\\
		-\D^2_{\hat{ \statefoll}}\left[\dynfoll(\hat{ \statefoll},{ \statefoll})({ \statefoll}-\hat{ \statefoll})   \right]\hat{ \psi}&-\D_{\hat{ \statefoll}}\left[\dynfoll(\hat{ \statefoll},{ \statefoll})({ \statefoll}-\hat{ \statefoll})   \right]
		\end{bmatrix}\\	
		\D_vG(\statelead,\hat{\statelead},\crtllead)=&\begin{bmatrix}
		\begin{matrix}
		0& &\\& \ddots &\\ & &0
		\end{matrix}\\
		-\D_\statefoll\tilde{m}(\statefoll)^\top \D_v\left[\nabla_m\objfoll(\tilde{m}(\statefoll),\crtllead)\right]
		\end{bmatrix}
\end{align*}
\endgroup
If in contrast, the OMO approach is chosen, the optimality system to  \eqref{Eq:Summary4} reads with $\fOMO=\fOMO(t,{\statelead})$ and $\Theta=\Theta(t,\statelead)$:
\begin{align}\label{Eq:Summary6}
\begin{split}
0=&\partial_t\fOMO+\div_\statelead \left(\fOMO \int G(\statelead,\hat{\statelead},v) \fOMOhat \,\d\hat{\statelead}   \right)\\
0=&\partial_t \Theta+\int\left[ G(\statelead,\hat{\statelead},v)^\top \nabla_\statelead\Theta+ G(\hat{\statelead},\statelead,v)^\top \nabla_{\hat{\statelead}}\hat{\Theta} \right]\fOMOhat\,\d\hat{\statelead}\\
&-\nabla_m\objlead(\crtllead,m_{\fOMO}(t))^\top\tilde{m}(\statefoll)-\frac{\reglead}{2}\crtllead^2\\
\crtllead(t)=&\frac{1}{\reglead} \left[ -\nabla_\crtllead\objlead(
\crtllead,m_{\fOMO}(t)) +\int \D_\crtllead G(\statelead,\hat{\statelead},\crtllead)\nabla_\statelead\Theta \fOMOhat\fOMO\,\d\hat{\statelead}\,\d\statelead\right]
\end{split}
\end{align}

\section*{Acknowledgments}

This work was supported by the DFG under Grant STE2063/2-1 and HE5386/19-1. 

\bibliographystyle{plain}
\bibliography{bibs/FiniteDimensionalGames.bib,bibs/General.bib,bibs/MeanField.bib}

\end{document}